\documentclass[reqno]{amsart}
\usepackage{fullpage}
\usepackage{array}
\usepackage{mathrsfs}
\usepackage{hyperref}
\usepackage{graphicx,color}
\usepackage{float}

\newcommand{\varied}

\usepackage{longtable}
\usepackage{booktabs}

\newcommand{\red}{\color{red}}
\newcommand{\black}{\color{black}}

\newtheorem{theorem}{Theorem}[section]

\newtheorem{lemma}[theorem]{Lemma}

\newtheorem{remark}[theorem]{Remark}
\newtheorem{proposition}[theorem]{Proposition}

\newcolumntype{C}[1]{>{\centering\let\newline\\\arraybackslash\hspace{0pt}}m{#1}}

\usepackage[utf8]{inputenc}
\usepackage{amsfonts}
\usepackage{amssymb}
\usepackage{verbatim}

\numberwithin{equation}{section}

\begin{document}
\title{Group gradings, graded identities and central polynomials of algebras with almost polynomial $\delta$-growth}

\author{Wesley Quaresma Cota$^{1,*}$ and Daniela Martinez Correa$^2$}
\thanks{{\it E-mail addresses:} wesleyqc@ime.usp.br and danielam.correa@ime.usp.br}

\thanks{\footnotesize $^{1}$ Supported by FAPESP, grants~2024/14914-9, 2025/05699-0 and 2026/06588-0}

\thanks{\footnotesize $^{2}$ Supported by FAPESP, grants~2024/14914-9, 2024/01338-0 and ~2025/11605-8.}

\thanks{\footnotesize $^{*}$ Corresponding author.}

\subjclass[2020]{Primary 16R10, 16W50, Secondary 16W55}

\keywords{ polynomial identities, central polynomials, graded algebras, codimension growth}

\dedicatory{Instituto de Matemática, Estatística e Ciência da Computação,\\
Universidade de São Paulo, São Paulo,
Brazil}

\begin{abstract}

 The finite-dimensional 
\(F\)-algebras
$D=F(e_{11}+e_{33})+Fe_{22}+Fe_{12}+Fe_{13}+Fe_{23}
\subseteq UT_3$ and
$D_0=Fe_{22}+Fe_{33}+Fe_{12}+Fe_{13}+Fe_{14}+Fe_{23}+Fe_{24}+Fe_{34}
\subseteq UT_4$
occupy a distinguished position in the theory of central polynomials.
Together with the infinite-dimensional Grassmann algebra, these
algebras generate the only varieties of almost polynomial growth
of proper central codimensions.
We determine finite generating sets for their $T$-spaces of central
polynomials and compute their central and proper central codimensions.
We also prove that every group grading on $D_0$ is isomorphic to an
elementary grading and determine the identities and codimensions
associated with all its $\mathbb{Z}_2$-gradings. 
\end{abstract}

\maketitle

\section{Introduction}

One of the fundamental problems in the theory of algebras with
polynomial identities is to determine explicit generators for the
\(T\)-ideal of identities of a given algebra. Even for classical
finite-dimensional algebras, this problem is highly nontrivial. For
instance, Drensky \cite{DK} proved that the \(T\)-ideal of polynomial
identities of \(M_2(F)\) is generated by the standard polynomial
\(\operatorname{St}_4(x_1,x_2,x_3,x_4)\) of degree $4$ and the Hall identity
$[[x_1,x_2]^2,x_3].$ By contrast, for the full matrix algebra \(M_n(F)\), \(n\geq3\), over a field of characteristic zero, no
explicit finite basis of identities is known in general. Kemer's celebrated finite basis theorem~\cite{Kemer1987}
asserts that, in characteristic zero, every proper $T$-ideal
of the free associative algebra is finitely generated. Nevertheless, the proof does not provide explicit generators for the identities of a prescribed algebra. 

An asymptotic approach to polynomial identities is provided by the
codimension sequence, introduced by Regev \cite{RG}. The codimension sequence $c_n(A)$
records the dimensions of the multilinear components of the
relatively free algebra of $A$. Regev proved that this sequence is
exponentially bounded for every PI-algebra, whereas Kemer
\cite{Kem} established the dichotomy between polynomially bounded
and exponential growth. The varieties generated by the
infinite-dimensional Grassmann algebra $\mathcal{G}$ and by $UT_2$
are precisely those of almost polynomial growth: their codimensions
grow exponentially, while those of every proper subvariety are
polynomially bounded. Here $UT_m$ denotes the algebra of
$m\times m$ upper triangular matrices over $F$.

Alongside polynomial identities, central polynomials have played a
fundamental role in PI-theory since the pioneering work of Formanek
\cite{Formanek} and Razmyslov \cite{Razmyslov} on the existence of proper central polynomials for matrix algebras. A polynomial
\(f\in F\langle X\rangle\) is called a \emph{central polynomial} of an algebra $A$ if it takes values in the center of $A$ whenever its variables are replaced by elements of $A$. The central polynomials of an algebra $A$ form a $T$-space,
denoted by $\operatorname{Id}^z(A)$, which need not be a $T$-ideal.
Their multilinear structure is studied through the central and
proper central codimensions, denoted by $c_n^z(A)$ and $\delta_n(A)$.
The varieties generated by $\mathcal{G}$ and $UT_2$ are also the only
varieties of almost polynomial growth of central codimensions
\cite{GiZai2}.

For proper central codimensions, Giambruno, La Mattina and Polcino Milies \cite{GiamDaniPolc} proved that the varieties of almost polynomial growth are precisely those generated by $\mathcal{G}$,
$D$ and $D_0$, where
\[ D= F(e_{11}+e_{33})+F e_{22}+F e_{12}+F e_{13}+F e_{23} \subseteq UT_3 \mbox{ and }\] \[ D_0= F e_{22}+F e_{33} +F e_{12}+F e_{13}+F e_{14} +F e_{23}+F e_{24}+F e_{34} \subseteq UT_4, \]
and $e_{ij}$ denotes the usual matrix unit.

Gordienko \cite{Gordienko} determined the ordinary polynomial
identities and codimensions of $D$. Gomez Parada and Koshlukov
\cite{Gomezparada} proved that every grading on $D$ by an abelian
group is isomorphic to an elementary grading. They also determined
its $\mathbb{Z}_2$-graded identities. The algebra \(D_0\), however, presents additional difficulties. This is an eight-dimensional nonunitary algebra and the study of its central polynomials, group gradings and graded identities requires a more delicate analysis.

In this paper, we prove that $\operatorname{Id}^z(D)$ and
$\operatorname{Id}^z(D_0)$ are $T$-ideals and determine finite
generating sets for them. We construct bases for the corresponding
multilinear quotients and obtain explicit formulas for their
central and proper central codimensions.

We also prove that every group grading on $D_0$ is isomorphic to an
elementary grading. We describe all its $\mathbb{Z}_2$-gradings and,
for each of them, determine generators for the
$T_{\mathbb{Z}_2}$-ideal of graded identities, bases for the
multilinear components of the relatively free graded algebra,
and the graded codimension sequence. The resulting eight
superalgebras are pairwise non-PI-equivalent as $\mathbb{Z}_2$-graded algebras.

\section{Preliminaries}

Throughout this paper, $F$ denotes a field of characteristic zero
and all algebras are associative $F$-algebras. Let $F\langle X\rangle$
be the free associative algebra without identity on the countable
set $X=\{x_1,x_2,\ldots\}$. A polynomial
$f=f(x_1,\ldots,x_n)$ is a \emph{ polynomial identity} of an algebra $A$ if
$$f(a_1,\ldots,a_n)=0\mbox{ for all }a_1,\ldots,a_n\in A.$$
The set $\operatorname{Id}(A)$ of all identities of $A$ is a
$T$-ideal, that is, an ideal invariant under all endomorphisms of
$F\langle X\rangle$. If $\operatorname{Id}(A)\neq\{0\}$, then
$A$ is a \emph{PI-algebra}.

A polynomial $f=f(x_1,\ldots,x_n)$ is a \emph{central polynomial} for $A$ if
$$f(a_1,\ldots,a_n)\in Z(A)\mbox{ for all } a_1,\ldots,a_n\in A,$$
where $Z(A)$ is the center of $A$. The central polynomials form a $T$-space $\operatorname{Id}^z(A)$, namely, a vector subspace invariant under all endomorphisms of $F\langle X\rangle$. In general, this space is not an ideal.
The elements of
$\operatorname{Id}^z(A)\setminus\operatorname{Id}(A)$
are called proper central polynomials.

For $S\subseteq F\langle X\rangle$, we write
$\langle S\rangle_T$ and $\langle S\rangle^T$ for the $T$-ideal and the $T$-space generated by $S$, respectively. We use $[u,v]=uv-vu$ and left-normed commutators
$[u_1,\ldots,u_m]=[[u_1,\ldots,u_{m-1}],u_m]$.

For $n\geq1$, let
\[
P_n=\operatorname{span}_F
\{x_{\sigma(1)}\cdots x_{\sigma(n)}\mid \sigma\in S_n\},
\]
where $S_n$ is the symmetric group.
Since $\operatorname{char}F=0$, every $T$-space is generated
by its multilinear elements.
The ordinary, central and proper central codimensions of $A$
are defined by
\[
c_n(A)=\dim_F\frac{P_n}{P_n\cap\operatorname{Id}(A)},
\qquad
c_n^z(A)=\dim_F\frac{P_n}{P_n\cap\operatorname{Id}^z(A)},
\qquad
\delta_n(A)=
\dim_F\frac{P_n\cap\operatorname{Id}^z(A)}
{P_n\cap\operatorname{Id}(A)}.
\]
These definitions give
\begin{equation}
c_n(A)=c_n^z(A)+\delta_n(A),\qquad n\geq1.
\end{equation}

The \emph{variety generated by $A$}, denoted by $\operatorname{var}(A)$,
consists of all algebras satisfying the identities of $A$.
Two algebras $A$ and $B$ are $T$-equivalent, written
$A\sim_T B$, if $\operatorname{Id}(A)=\operatorname{Id}(B)$,
or equivalently,
$\operatorname{var}(A)=\operatorname{var}(B)$. Since $f\in\operatorname{Id}^z(A)$ if and only if $
[f,x_{n+1}]\in\operatorname{Id}(A),$
where $x_{n+1}$ does not occur in $f$,
the space $\operatorname{Id}^z(A)$ depends only on
$\operatorname{Id}(A)$.
For $\mathcal V=\operatorname{var}(A)$, we may therefore set
$c_n(\mathcal V)=c_n(A)$,
$c_n^z(\mathcal V)=c_n^z(A)$ and
$\delta_n(\mathcal V)=\delta_n(A)$.

Regev's theorem \cite{RG} and the inequalities
$c_n^z(A),\delta_n(A)\leq c_n(A)$ imply that all three
sequences are exponentially bounded for every PI-algebra.
Their exponential growth is measured by
\[
\exp(A)=\lim_{n\to\infty}\sqrt[n]{c_n(A)},
\qquad
\exp^z(A)=\lim_{n\to\infty}\sqrt[n]{c_n^z(A)},
\qquad
\exp^\delta(A)=\lim_{n\to\infty}\sqrt[n]{\delta_n(A)},
\]
called the  \emph{ordinary},  \emph{central} and  \emph{proper central exponents}, respectively. For PI-algebras, these limits exist and are nonnegative integers \cite{GiZai2, GiZai}.

An algebra $A$ has \emph{almost polynomial growth of codimensions} if
$c_n(A)$ grows exponentially, whereas $c_n(\mathcal W)$ is
polynomially bounded for every proper subvariety
$\mathcal W$ of $\operatorname{var}(A)$.
Almost polynomial growth of central and proper central codimensions is defined analogously using $c_n^z$ and $\delta_n$, respectively.

In what follows, we denote by $
\mathcal{G}
=
\langle 1,e_1,e_2,\ldots
\mid e_i e_j=-e_j e_i;\ i,j\geq 1\rangle$ the infinite-dimensional unitary Grassmann algebra
over \(F\).

\begin{theorem}[Kemer \cite{Kem}]\label{kemersresult}
Let \(A\) be a PI-algebra such that \(\exp(A)\geq2\). Then either $\operatorname{Id}(A)\subseteq\operatorname{Id}(UT_2)$ or $\operatorname{Id}(A)\subseteq\operatorname{Id}(\mathcal G).$ Consequently, \(\textnormal{var}(UT_2)\) and \(\textnormal{var}(\mathcal G)\) are the only varieties of almost polynomial growth of ordinary codimensions.
\end{theorem}

Kemer also proved that the codimension sequence of a PI-algebra is either polynomially bounded or grows exponentially.
In particular, no intermediate growth can occur. The corresponding result for central codimensions yields the same two varieties of
almost polynomial growth, that is, \(\operatorname{var}(UT_2)\) and
\(\operatorname{var}(\mathcal G)\); see \cite{GiZai2}.

For later use, we  recall the polynomial identities, central polynomials and
codimension sequences of these two algebras.

\begin{proposition} \cite{Malvec}\label{prop:UT2-central}
The algebra \(UT_2\) has no proper central polynomials and
\[
\operatorname{Id}(UT_2)=\operatorname{Id}^z(UT_2)
=\langle[x_1,x_2][x_3,x_4]\rangle_T.
\]
Moreover,
$c_n(UT_2)=c_n^z(UT_2)=2^{n-1}(n-2)+2$
and $\delta_n(UT_2)=0$ for every \(n\geq1\).
\end{proposition}

We shall also use the following description of the Grassmann algebra.

\begin{proposition}[\cite{Brandao}]\label{prop:Grassmann-central}
For the algebra \(\mathcal G\), the following statements hold:
\begin{enumerate}
    \item[\textnormal{\(1\)}]
    \(\operatorname{Id}(\mathcal G)
      =\langle[x_1,x_2,x_3]\rangle_T\).

    \item[\textnormal{\(2\)}]
    \(\operatorname{Id}^z(\mathcal G)
      =\langle x_1[x_2,x_3,x_4],[x_1,x_2]\rangle^T\).

    \item[\textnormal{\(3\)}]
    \(c_n(\mathcal G)=2^{n-1}\) and
    \(c_n^z(\mathcal G)=\delta_n(\mathcal G)=2^{n-2}\), for every
    \(n\geq2\).
\end{enumerate}
\end{proposition}

For the algebras $D$ and $D_0$ defined in the introduction, the corresponding result for proper central codimensions is the following.

\begin{theorem}[\cite{GiamDaniPolc}]\label{thm:proper-central-growth}
Let \(A\) be a PI-algebra such that \(\exp^\delta(A)\geq2\). Then either
$\operatorname{Id}(A)\subseteq\operatorname{Id}(D)$ or $\operatorname{Id}(A)\subseteq\operatorname{Id}(D_0)$ or $\operatorname{Id}(A)\subseteq\operatorname{Id}(\mathcal G).$ Moreover, \(\textnormal{var}(D)\), \(\textnormal{var}(D_0)\) and \(\textnormal{var}(\mathcal G)\) are
the only varieties of almost polynomial growth of proper central codimensions.
\end{theorem}

The polynomial--exponential growth dichotomy also holds for proper central codimensions: the sequence of a PI-algebra is either polynomially bounded or grows exponentially.

\subsection{Graded identities and graded codimensions}

Let \(G\) be a multiplicative group with identity element \(1_G\). An algebra
\(A\) is said to be \emph{\(G\)-graded} if it admits a vector space
decomposition
\[
A=\bigoplus_{g\in G}A_g
\]
such that \(A_gA_h\subseteq A_{gh}\) for all \(g,h\in G\). The subspace
\(A_g\) is called the \emph{homogeneous component of degree \(g\)} and a
nonzero element \(a\in A_g\) is said to be \emph{homogeneous of degree \(g\)};
in this case, we write \(\deg_G a=g\).
For each $g\in G$, take a countable set
$X_g=\{x_i^g\mid i\geq1\}$, with these sets pairwise disjoint,
and put $X^G=\bigcup_{g\in G}X_g$.
The free associative algebra $F\langle X^G\rangle$ is
$G$-graded by
\[
\deg_G(x_{i_1}^{g_1}\cdots x_{i_m}^{g_m})
=g_1\cdots g_m.
\]
A polynomial $f=f(x_1^{g_1},\ldots,x_m^{g_m})$ is a \emph{graded
identity of $A$} if $f(a_1,\ldots,a_m)=0$ whenever
$a_i\in A_{g_i}$.
The graded identities form a \emph{$T_G$-ideal}
$\operatorname{Id}^G(A)$, that is, an ideal invariant under
all graded endomorphisms of $F\langle X^G\rangle$.
We denote by $\langle S\rangle_{T_G}$ the $T_G$-ideal
generated by a set $S$.

Assume that $G$ is finite. For $n\geq1$, set
\[
P_n^G=\operatorname{span}_F
\left\{
x_{\sigma(1)}^{g_{\sigma(1)}}\cdots
x_{\sigma(n)}^{g_{\sigma(n)}}
\ \middle|\
\sigma\in S_n,\ g_1,\ldots,g_n\in G
\right\}.
\]
The $n$th graded codimension of $A$ is
\[
c_n^G(A)=
\dim_F\frac{P_n^G}{P_n^G\cap\operatorname{Id}^G(A)}.
\]
For every $n\geq1$,
\[
c_n(A)\leq c_n^G(A)\leq |G|^n c_n(A);
\]
see \cite{GiamReg}.
Thus the graded codimension sequence of a PI-algebra is
exponentially bounded.

For $G=\mathbb Z_2$, we write $F\langle Y\cup Z\rangle$, where $X_{\overline{0}}=Y=\{y_1,y_2,\ldots\}$ consists of the variables of degree $0$ and $X_{\overline{1}}=Z=\{z_1,z_2,\ldots\}$ consists of the variables of degree $1$. We also say that the variables in $Y$ are even and those of $Z$ are odd. 

\begin{remark}\label{reflectionidentities}
Let $\rho:UT_n\to UT_n$ be the anti-automorphism defined by $\rho(e_{ij})=e_{n-j+1,n-i+1},$ for all $1\leq i\leq j\leq n.$ If $B=B_{0}\oplus B_{1}$ is a $\mathbb Z_2$-graded
subalgebra with $\rho(B)=B$, let $B^\rho$ denote the same algebra with homogeneous components
$(B^\rho)_{\bar\imath}=\rho(B_{\bar\imath})$.

For $f\in F\langle Y\cup Z\rangle$, let $f^*$ be obtained
by reversing every monomial of $f$. Then
$f\in\operatorname{Id}^{\mathbb Z_2}(B)$ if and only if $f^*\in\operatorname{Id}^{\mathbb Z_2}(B^\rho).$
Consequently, $\operatorname{Id}^{\mathbb Z_2}(B)
=\langle f_1,\ldots,f_k\rangle_{T_{\mathbb Z_2}}$ if and only if $\operatorname{Id}^{\mathbb Z_2}(B^\rho)
=\langle f_1^*,\ldots,f_k^*\rangle_{T_{\mathbb Z_2}},$ and
$c_n^{\mathbb Z_2}(B)=c_n^{\mathbb Z_2}(B^\rho)$
for every $n\geq1$.
\end{remark}

\section{Central polynomials of $D$ and $D_0$}

We now determine the central polynomials and the central and proper central codimensions of $D$ and $D_0$. First, we study the central polynomials of the algebra
\[
D=
\left\{
\begin{pmatrix}
a & c & d\\
0 & b & e\\
0 & 0 & a
\end{pmatrix}
\mid 
a,b,c,d,e\in F
\right\}.
\] Set $p=[x_1,x_2][x_3,x_4].$

\begin{theorem}
\label{thm:central-D}
For the algebra $D$, the following statements hold.
\begin{enumerate}
\item
The space of central polynomials of $D$ is the $T$-ideal $\operatorname{Id}^z(D)
=
\operatorname{Id}(UT_2)
=
\langle p\rangle_T.$

\item
For every $n\geq 4$, $c_n^z(D)=2^{n-1}(n-2)+2$ and
$\delta_n(D)=2^{n-2}(n^2-5n+4)+n-1+\binom{n}{2}-\binom{n}{3}.$\end{enumerate}
\end{theorem}

\begin{proof}
Let $I=\langle p\rangle_T$. Since
$[D,D][D,D]\subseteq Fe_{13}$ and $Fe_{13}$ is a central
two-sided ideal of $D$, we have $D[D,D][D,D], [D,D][D,D]D\subseteq Fe_{13}$. Thus, every polynomial in $I$ takes its
values in $Fe_{13}$ and so $I\subseteq\operatorname{Id}^z(D)$.

Conversely, consider the surjective algebra homomorphism $\pi:D\rightarrow UT_2$ given by
\[
\begin{pmatrix}
a&c&d\\
0&b&e\\
0&0&a
\end{pmatrix}
\mapsto
\begin{pmatrix}
a&c\\
0&b
\end{pmatrix}.
\]
Since $\pi(Z(D))\subseteq Z(UT_2)$, every central polynomial
of $D$ is central on $UT_2$. Thus,
\[
\operatorname{Id}^z(D)
\subseteq
\operatorname{Id}^z(UT_2)
=
\operatorname{Id}(UT_2)
=
I.
\]
This proves the $T$-ideal description in item~$(1)$. Consequently, $c_n^z(D)=c_n(UT_2)= 2^{n-1}(n-2)+2$ and $\delta_n(D)=c_n(D)-c_n^z(D)=2^{n-2}(n^2-5n+4)+n-1+\binom{n}{2}-\binom{n}{3}$, which completes the proof. 
\end{proof}

Now, we consider the algebra
\[
D_0=
\left\{
\begin{pmatrix}
0 & c & d & e\\
0 & a & f & g\\
0 & 0 & b & h\\
0 & 0 & 0 & 0
\end{pmatrix}
\mathrel{\Big|}
a,b,c,d,e,f,g,h\in F
\right\}.
\]
 We begin by recalling a result that will be used throughout this section.

\begin{theorem} \cite{GiamDaniPolc}  \label{identidadesd0}\label{identidadesordinarias} For the algebra $D_0$ we have:
    \begin{enumerate}
        \item[(1)] $\textnormal{Id}(D_0)= \langle x_1[x_2,x_3][x_4,x_5]x_6 \rangle_T$.
        \item[(2)] For every $n\geq 3$, a basis of $P_n$ modulo
$P_n\cap\operatorname{Id}(D_0)$ is given by
\[
x_a x_{i_1}\cdots x_{i_{n-2}}x_b
\quad\text{and}\quad
x_a x_{i_1}\cdots x_{i_r}
[x_{j_1},\ldots,x_{j_s}]x_b.
\]
In the first family, $i_1<\cdots<i_{n-2}$.
In the second family, $r\geq 0$, $s\geq 2$, $r+s=n-2$,
$i_1<\cdots<i_r$ and $j_1>j_2<\cdots<j_s$.
In both families, all indices are distinct and together form
$\{1,\ldots,n\}$.
        \item[(3)] $c_n(D_0)=2^{n-3}n(n-1)(n-4)+2n(n-1)$, for $n\geq 7$. 
    \end{enumerate}
\end{theorem}

We now turn to the central polynomials of $D_0$. Since
$Z(D_0)=Fe_{14}$ and $D_0e_{14}=e_{14}D_0=0$, every
central polynomial $f$ of $D_0$ satisfies
$fx_{n+1},x_{n+1}f\in\operatorname{Id}(D_0)$, where
$x_{n+1}$ is a new variable. By substituting arbitrary
polynomials for $x_{n+1}$, it follows that the $T$-space
$\operatorname{Id}^z(D_0)$ is closed under left and right
multiplication in $F\langle X\rangle$ and is therefore
a $T$-ideal.

Suppose that $0\ne f\in P_n\cap\operatorname{Id}^z(D_0)$
with $n\leq4$. Then $fx_{n+1}$ is a nonzero identity of
degree $n+1\leq5$. This contradicts
\[
\operatorname{Id}(D_0)
=
\langle x_1[x_2,x_3][x_4,x_5]x_6\rangle_T,
\]
since every nonzero element of this $T$-ideal has degree
at least six in the nonunital free algebra. Therefore,
$P_n\cap\operatorname{Id}^z(D_0)=0$ for $1\leq n\leq4$.
Moreover,
\[
\operatorname{St}_5(x_1,\ldots,x_5)
\quad\text{and}\quad
[x_1,x_2][x_3,x_4][x_5,x_6]
\]
take their values in $Fe_{14}$. Since
$\operatorname{Id}^z(D_0)$ is a $T$-ideal containing
these polynomials and $\operatorname{Id}(D_0)$, we obtain
\[
I=
\left\langle
x_1[x_2,x_3][x_4,x_5]x_6,\,
\operatorname{St}_5(x_1,\ldots,x_5),\,
[x_1,x_2][x_3,x_4][x_5,x_6]
\right\rangle_T
\subseteq\operatorname{Id}^z(D_0).
\]

 By the standard basis for the relatively free algebra of $UT_3$ (see \cite[Remark~10]{Drensky}), every polynomial in $P_n$ can be written, modulo $P_n\cap I$, as a linear combination of the polynomial $x_{1}\ldots x_n$ and polynomials belonging to the following four families:
\begin{enumerate}
    \item \label{1} $[x_{j_1}, x_{i_2}, \ldots, x_{i_{n}} ]$, with $j_1>i_2<\cdots <i_n $;

    \item \label{2} $x_{i_1}\cdots x_{i_k}[x_{j_1},x_{t_2}, \ldots, x_{t_{n-k}}]$, with $k\geq 1$, $n-k\geq 2$, $i_1<\cdots < {i_k}$, $j_1>t_2<
    \cdots <t_{n-k}$;

    \item \label{3} $[x_{j_1}, x_{i_2}, \ldots, x_{i_k}][x_{j_2}, x_{t_2}, \ldots , x_{t_{n-k}}]$, with $k\geq 2$, $n-k\geq 2$, $j_1>i_2<\cdots < i_{k}$, $j_2>t_2<\cdots <t_{n-k}$;

    \item[(4)]  $x_{l_1}\ldots x_{l_r}[x_{j_1}, x_{i_2}, \ldots, x_{i_k}][x_{j_2}, x_{t_2}, \ldots , x_{t_{n-k-r}}]$, with $r\geq 1,$ $k\geq 2$, $n-k-r\geq 2$, $l_1<\cdots <l_r$, $j_1>i_2<\cdots < i_{k}$, $j_2>t_2<\cdots < t_{n-k-r}$.
\end{enumerate}

In each case, all the indices occurring in the polynomial are distinct and belong to $\{1, \ldots, n\}$. For $s\in\{1,2,3,4\}$, let $\mathcal I_s$ denote the set of admissible choices of indices satisfying the conditions of family $(s)$ above. For $a\in\mathcal I_s$, we denote by $m_a$ the corresponding polynomial of type $(s)$.

The next result separates the first three families from the genuinely new part of the problem. 

\begin{lemma}\label{lemtype123}
Let $f=\alpha x_1\cdots x_n+f_1+f_2+f_3+f_4\in P_n,
$ where $f_i$ is a linear combination of polynomials of type $(i)$, for $i\in \{1,2,3,4\}$. If $
f\in\operatorname{Id}^z(D_0),$
then $$\alpha= f_1=f_2=f_3=0.$$
Consequently, the polynomial $x_1\cdots x_n$ and the polynomials of types $(1)$, $(2)$ and $(3)$ are linearly independent modulo $P_n\cap\operatorname{Id}^z(D_0)$ and remain linearly independent modulo the span of the polynomials of type $(4)$.
\end{lemma}  

\begin{proof}
Suppose that
$f=\alpha x_1\cdots x_n+f_1+f_2+f_3+f_4
\in\operatorname{Id}^z(D_0)$,
where, for $i\in\{1,2,3,4\}$, the polynomial $f_i$ is a linear
combination of polynomials of type $(i)$. Thus,
$$
f=\alpha x_1\cdots x_n+
\sum_{a\in\mathcal I_1}\alpha_a m_a
+
\sum_{b\in\mathcal I_2}\beta_b m_b
+
\sum_{c\in\mathcal I_3}\gamma_c m_c
+
\sum_{d\in\mathcal I_4}\sigma_d m_d.
$$
We shall prove successively that $\alpha=0$ and all coefficients
$\beta_b$, $\alpha_a$ and $\gamma_c$ are zero.

Set $h=e_{22}+e_{33}$. By evaluating all variables at $h$, we
obtain $\alpha h\in Z(D_0)=Fe_{14}$ and hence $\alpha=0$.

Now, we consider the polynomials of type $(2)$. Fix
$$
m_b=x_{i_1}\cdots x_{i_k}
[x_{j_1},x_{t_2},\ldots,x_{t_{n-k}}],
$$
and set $L=\{i_1,\ldots,i_k\}$. Consider the evaluation
$\varphi$ defined by
$$
\varphi(x_{i_q})=h,\quad 1\leq q\leq k,\qquad
\varphi(x_{j_1})=e_{23}
\quad\mbox{and}\quad
\varphi(x_{t_q})=e_{33},\quad 2\leq q\leq n-k.
$$
Since $he_{23}=e_{23}$ and $[e_{23},e_{33}]=e_{23}$,
we obtain $\varphi(m_b)=e_{23}$.

We claim that every other polynomial occurring in $f$ is
annihilated by $\varphi$.

Indeed, a polynomial of type $(1)$ contains all variables
inside a single commutator. Since at least one variable is
evaluated at $h$, which commutes with every assigned element,
its value is zero. Moreover, every polynomial of type $(3)$
or $(4)$ contains two commutators, whereas $e_{23}$ is the
only non-diagonal value occurring in the evaluation.
Hence at least one of the two commutators is evaluated
entirely on diagonal elements and is therefore zero.

Now consider another polynomial of type $(2)$ with a nonzero
value. The variable $x_{j_1}$ must occur inside its commutator,
since all other assigned elements are diagonal. No variable
evaluated at $h$ can occur inside that commutator.
Conversely, no variable evaluated at $e_{33}$ can occur
outside it, since the commutator evaluates to a multiple of
$e_{23}$ and $e_{33}e_{23}=0$.
Hence its outside variables are exactly those indexed by $L$.

The variables of its commutator are therefore
$x_{j_1},x_{t_2},\ldots,x_{t_{n-k}}$.
The smallest index is $t_2$ and it must occur in the second
position. Since $t_2<j_1$, a nonzero first bracket forces
$x_{j_1}$ to be the first entry. The increasing order of the
remaining indices then determines precisely $m_b$.
Consequently, $\varphi(f)=\beta_b e_{23}$.
Since $f$ is central and $e_{23}\notin Z(D_0)$, we obtain
$\beta_b=0$. As $b$ was arbitrary, it follows that $f_2=0$.

We now consider the polynomials of type $(1)$. Fix
$m_a=[x_{j_1},x_{i_2},\ldots,x_{i_n}]$.
Consider the evaluation $\psi$ defined by
$\psi(x_{j_1})=e_{23}$ and
$\psi(x_r)=e_{33}$ for $r\neq j_1$.
Then
$\psi(m_a)=[e_{23},e_{33},\ldots,e_{33}]=e_{23}$.

Every commutator of type $(1)$ involves all $n$ variables,
with the smallest index in the second position.
Since $i_2<j_1$, a nonzero first bracket forces $x_{j_1}$
to be the first entry. The increasing order of the remaining
indices then determines precisely $m_a$.
Every polynomial of type $(3)$ or $(4)$ also vanishes
under $\psi$, since there is only one non-diagonal value
available for two commutators.
Since $\alpha=0$ and $f_2=0$, it follows that
$\psi(f)=\alpha_a e_{23}$.
The centrality of $f$ implies $\alpha_a=0$.
As $a$ was arbitrary, we obtain $f_1=0$.

It remains to consider the polynomials of type $(3)$. Fix
$$
m_c=
[x_{j_1},x_{i_2},\ldots,x_{i_k}]
[x_{j_2},x_{t_2},\ldots,x_{t_{n-k}}].
$$
Consider the evaluation $\theta$ defined by
$$
\theta(x_{j_1})=e_{12},\qquad
\theta(x_{i_q})=h,\quad 2\leq q\leq k,\qquad
\theta(x_{j_2})=e_{23}
\quad\mbox{and}\quad
\theta(x_{t_q})=e_{33},\quad 2\leq q\leq n-k.
$$
Since $[e_{12},h]=e_{12}$ and $[e_{23},e_{33}]=e_{23}$,
we have $\theta(m_c)=e_{12}e_{23}=e_{13}$.

We show that every other polynomial of type $(3)$ vanishes
under this evaluation. Since the only non-diagonal values are
$e_{12}$ and $e_{23}$, a nonzero product of the two commutators
must have the first evaluated to a nonzero multiple of
$e_{12}$ and the second to a nonzero multiple of $e_{23}$.

No variable evaluated at $e_{33}$ can occur in the first
commutator, because $[e_{12},e_{33}]=0$.
Similarly, no variable evaluated at $h$ can occur in the
second commutator, because $[e_{23},h]=0$.
Thus the variables of the first commutator are precisely
$x_{j_1},x_{i_2},\ldots,x_{i_k}$ and those of the second are
$x_{j_2},x_{t_2},\ldots,x_{t_{n-k}}$.

In each commutator, the smallest index occurs in the second
position. Since $i_2<j_1$ and $t_2<j_2$, a nonzero first
bracket forces $x_{j_1}$ and $x_{j_2}$ to be the respective
first entries. The increasing order of the remaining
indices determines both commutators uniquely.
Therefore the only polynomial of type $(3)$ having a
nonzero value is $m_c$.

Finally, consider a polynomial of type $(4)$. It contains
a nonempty monomial preceding its two commutators.
For both commutators to be nonzero, the variables evaluated
at $e_{12}$ and $e_{23}$ must occur in different commutators.
Their product can be nonzero only when the first commutator
evaluates to a multiple of $e_{12}$ and the second to a
multiple of $e_{23}$.
Consequently, all variables in the initial monomial are
evaluated at diagonal elements belonging to $Fe_{22}+Fe_{33}$.
Since $(Fe_{22}+Fe_{33})e_{12}=0$, every polynomial of
type $(4)$ vanishes under $\theta$.

Since $\alpha=0$ and $f_1=f_2=0$, we conclude that
$\theta(f)=\gamma_c e_{13}$.
As $f$ is central and $e_{13}\notin Z(D_0)$, it follows
that $\gamma_c=0$.
As $c$ was arbitrary, we obtain $f_3=0$.

Therefore, $\alpha=0$ and $f_1=f_2=f_3=0$, as required.
\end{proof}

Lemma~\ref{lemtype123} shows that the first three families are linearly independent, modulo $P_n\cap \textnormal{Id}^z(D_0)$, and their span has trivial intersection with the span of the residue classes of the type $(4)$ polynomials.
Hence it remains to study the polynomials of type $(4)$.

We first reduce these polynomials to products involving two simple commutators. Throughout the remainder of the section, congruences are taken modulo $$I=\left\langle
x_1[x_2,x_3][x_4,x_5]x_6,\,
St_5(x_1,\ldots,x_5),\,
[x_1,x_2][x_3,x_4][x_5,x_6]
\right\rangle_T.
$$ We use the following consequence repeatedly. For arbitrary
words $u_0,u_1,u_2,u_3$ and arbitrary polynomials
$a_1,\ldots,a_6$, we have
\[
u_0[a_1,a_2]u_1[a_3,a_4]u_2[a_5,a_6]u_3\in I.
\]
Indeed, this follows by induction on the lengths of the intervening
words from the identity
$[a,b]x=[a,bx]-b[a,x]$ and from the fact that
$[x_1,x_2][x_3,x_4][x_5,x_6]\in I.$

\begin{lemma}\label{lem}
Every polynomial of type $(4)$ is congruent, modulo $I$, to a linear combination of polynomials of the form
$$
x_{l_1}\cdots x_{l_r}
[x_{j_1},x_{i_1}]
x_{t_1}\cdots x_{t_s}
[x_{j_2},x_{i_2}],
$$
where $r\geq1$, $s\geq0$,
$l_1<\cdots<l_r$, $t_1<\cdots<t_s$,
$j_1>i_1$ and $j_2>i_2$.
\end{lemma}

\begin{proof}
Let
$f=x_{l_1}\cdots x_{l_r}
[x_{j_1},x_{i_2},\ldots,x_{i_k}]
[x_{j_2},x_{t_2},\ldots,x_{t_m}]$
be a polynomial of type $(4)$.
We first reduce the second commutator. If $m>2$, write
$[x_{j_2},x_{t_2},\ldots,x_{t_m}]=[C,x_{t_m}]$,
where $C=[x_{j_2},x_{t_2},\ldots,x_{t_{m-1}}]$.
Then
\begin{align*}
f
&=x_{l_1}\cdots x_{l_r}
[x_{j_1},x_{i_2},\ldots,x_{i_k}]Cx_{t_m}-
x_{l_1}\cdots x_{l_r}
[x_{j_1},x_{i_2},\ldots,x_{i_k}]x_{t_m}C.
\end{align*}
The first summand belongs to $I$, since $I$ is a $T$-ideal
containing $x_1[x_2,x_3][x_4,x_5]x_6$. Therefore,
$$
f\equiv
-x_{l_1}\cdots x_{l_r}
[x_{j_1},x_{i_2},\ldots,x_{i_k}]x_{t_m}C
\pmod I.
$$
The same reduction remains valid when a monomial has already
been inserted between the commutators. Indeed, for commutators
$B,C$, nonempty monomials $X,U$ and a variable $x$, we have
$$
XBUCx=XUBCx+X[B,U]Cx\in I,
$$
since both summands belong to $I$.
For empty $U$, the inclusion follows directly from that generator.

Repeating this procedure, we reduce the second commutator to
length two. Hence $f$ is congruent modulo $I$, up to sign, to a
polynomial of the form
$X[x_{j_1},x_{i_2},\ldots,x_{i_k}]U[x_{j_2},x_{i_2'}]$,
where $X$ is a nonempty monomial and $U$ is a possibly empty
monomial.

If $k>2$, write
$[x_{j_1},x_{i_2},\ldots,x_{i_k}]=[A,x_{i_k}]$,
with $A=[x_{j_1},x_{i_2},\ldots,x_{i_{k-1}}]$.
Using $[A,x_{i_k}]=Ax_{i_k}-x_{i_k}A$, we obtain
\begin{align*}
X[A,x_{i_k}]U[x_{j_2},x_{i_2'}]
=XAx_{i_k}U[x_{j_2},x_{i_2'}]-Xx_{i_k}AU[x_{j_2},x_{i_2'}].
\end{align*}
Both summands contain a first commutator of smaller length
and retain a nonempty initial monomial. Repeating this
expansion, we obtain a linear combination of polynomials of
the form
$X[x_{j_1},x_{i_1}]U[x_{j_2},x_{i_2}]$.
By antisymmetry of the commutator, we may assume $j_1>i_1$
and $j_2>i_2$.

Finally, since $[x_1,x_2][x_3,x_4][x_5,x_6]\in I$, we may
assume, modulo $I$, that $X=x_{l_1}\cdots x_{l_r}$, with
$l_1<\cdots<l_r$ and $U=x_{t_1}\cdots x_{t_s}$, with
$t_1<\cdots<t_s$. This gives the required form.
\end{proof}

We shall call the polynomials obtained in Lemma~\ref{lem} polynomials of type $(4')$. Since the subsequent reductions depend on the number of variables occurring between the two commutators, it is convenient to separate them according to the length of the middle block.

For each integer $s$ with $0\leq s\leq n-5$, let $\mathcal M_{n,s}^{(4)}$ denote the subspace of $P_n$ spanned by all polynomials of the form
$x_{l_1}\cdots x_{l_r}
[x_{j_1},x_{i_1}]
x_{t_1}\cdots x_{t_s}
[x_{j_2},x_{i_2}],$ where $r=n-s-4\geq 1$, $ 
l_1<\cdots<l_r$,
$t_1<\cdots<t_s$, $j_1>i_1$ and $j_2>i_2.$ Moreover, all the indices are distinct and satisfy $\{l_1,\ldots,l_r,j_1,i_1,t_1,\ldots,t_s,j_2,i_2\}
=
\{1,\ldots,n\}.$

We denote by $\mathcal M_n^{(4)}$ the subspace generated by all polynomials of type $(4')$. Hence
\begin{equation*}
\mathcal M_n^{(4)}
=
\mathcal M_{n,0}^{(4)}
+
\sum_{s=1}^{n-5}\mathcal M_{n,s}^{(4)}.
\end{equation*}

At this stage, we do not claim that the above sum is linearly independent modulo $I$.  We first study $\mathcal M_{n,0}^{(4)}$, where the two commutators are adjacent and then consider the spaces $\mathcal M_{n,s}^{(4)}$ with $s\geq1$.

\subsection{The case $\mathcal M_{n,0}^{(4)}$} In this case, the standard polynomial of degree five gives us
\begin{equation} \label{standard5}
St_5(x_1,x_2,x_3,x_4,x_5)
=\frac{1}{4}
\sum_{\sigma\in S_5}
\operatorname{sgn}(\sigma)
x_{\sigma(1)}
[x_{\sigma(2)},x_{\sigma(3)}]
[x_{\sigma(4)},x_{\sigma(5)}]    
\end{equation}

After orienting the two commutators, the previous equation expresses $St_5$ as a signed sum of polynomials
$x_l[x_j,x_i][x_q,x_p],$
obtained by distributing the five variables among the outside position and the two oriented commutators. In particular, if $a<b<c<d<e,$
then the polynomial $x_e[x_d,x_c][x_b,x_a]
$ can be expressed, modulo $I$, as a linear combination of the other polynomials of the same form. We shall use this consequence repeatedly. We now define two disjoint families in $\mathcal M_{n,0}^{(4)}$:
$$
\mathcal B'
=
\left\{
m\in\mathcal M_{n,0}^{(4)}
\mid
i_1<j_2,\quad j_1\neq n
\right\}\quad \mbox{ and }\quad \mathcal B''
=
\left\{
m\in\mathcal M_{n,0}^{(4)}
\mid
j_1=n
\right\}.
$$

\begin{lemma}\label{lem:no-gap-generation}
The family $\mathcal B'\cup\mathcal B''$ generates $\mathcal M_{n,0}^{(4)}$ modulo $I$. Moreover,
$\left|\mathcal B'\cup\mathcal B''\right|
=
6\binom{n}{4}
-
\binom{n-1}{4}.
$
\end{lemma}

\begin{proof}
Let
$m=
x_{l_1}\cdots x_{l_{n-4}}
[x_{j_1},x_{i_1}]
[x_{j_2},x_{i_2}]$
be one of the polynomials spanning $\mathcal M_{n,0}^{(4)}$.
If $j_1=n$, then $m\in\mathcal B''$.
Assume therefore that $j_1\neq n$.
If $i_1<j_2$, then $m\in\mathcal B'$.
It remains to consider the case
$j_1\neq n$ and $i_1>j_2$.

Since $j_1>i_1>j_2>i_2$, the index $n$ cannot occur
in either commutator. Hence $l_{n-4}=n$.
Write $a=i_2$, $b=j_2$, $c=i_1$ and $d=j_1$.
Then $a<b<c<d<n$ and $m$ has the form
$Xx_n[x_d,x_c][x_b,x_a]$,
where $X$ is the possibly empty product of the remaining
outside variables.

By equation~(\ref{standard5}),
$x_n[x_d,x_c][x_b,x_a]$ is congruent modulo $I$ to a
linear combination of the other twenty-nine polynomials
obtained from the variables $x_a,x_b,x_c,x_d,x_n$,
since its coefficient in that expansion is $\pm1$.
Multiply this congruence on the left by $X$ and order
the outside variables modulo $I$.

Consider now the other twenty-nine terms in the expansion of
$\operatorname{St}_5(x_a,x_b,x_c,x_d,x_n)$.
Unlike the original term, these may have $x_n$ inside either
commutator. After multiplication by $X$ and ordering the
outside variables, a term with $x_n$ in the first commutator
belongs to $\mathcal B''$, whereas a term with $x_n$ in the
second belongs to $\mathcal B'$.

Among the terms with $x_n$ outside both commutators, the only
one satisfying $i_1>j_2$ is
$x_n[x_d,x_c][x_b,x_a]$, which has already been isolated.
Thus all the remaining terms, after multiplication by $X$,
belong to $\mathcal B'\cup\mathcal B''$.
Therefore $\mathcal B'\cup\mathcal B''$ generates
$\mathcal M_{n,0}^{(4)}$ modulo $I$.

It remains to count its elements. Choose the four variables
appearing in the two commutators.
There are $\binom{n}{4}$ possible choices.
For each fixed set of four variables, there are six ways
to choose the two variables of the first commutator;
the orientation of each commutator and the increasing
order of the outside variables are then uniquely determined.
Hence there are $6\binom{n}{4}$ polynomials of this form.

Such a polynomial does not belong to
$\mathcal B'\cup\mathcal B''$ precisely when
$j_1\neq n$ and $i_1>j_2$.
As observed above, this forces $n$ to be an outside variable.
For every choice
$
\{a<b<c<d\}\subseteq\{1,\ldots,n-1\},
$
there is exactly one such polynomial, namely
$x_{l_1}\cdots x_{l_{n-5}}x_n
[x_d,x_c][x_b,x_a]$.
Therefore the number of excluded polynomials is
$\binom{n-1}{4}$.
Consequently,
$$
\left|\mathcal B'\cup\mathcal B''\right|
=
6\binom{n}{4}
-
\binom{n-1}{4}.
$$
This completes the proof.
\end{proof}

We next turn to the spaces $\mathcal M_{n,s}^{(4)}$ with $s\geq1$, where variables occur between the two commutators. \black

\subsection{The case $\mathcal M_{n,s}^{(4)}$, $s\geq 1$.} The first step is to show that, modulo $I$ and the cases already considered, we may assume that the first entry of the first commutator is larger than every variable occurring in the middle block.

We shall use the polynomial
\begin{align*}
g={}&
-x_5x_6[x_4,x_2][x_3,x_1]
+x_4x_5[x_6,x_2][x_3,x_1]
-x_2x_5[x_6,x_4][x_3,x_1] \\
&+
x_5[x_4,x_2]x_6[x_3,x_1]
-x_5[x_6,x_2]x_4[x_3,x_1]
+x_5[x_6,x_4]x_2[x_3,x_1].
\end{align*}

We prove that $g$ is already a consequence of the central polynomial
$q=[x_1,x_2][x_3,x_4][x_5,x_6].$

\begin{lemma}\label{lem:g-consequence}
The polynomial $g$ belongs to $\langle q\rangle^T$. Consequently, modulo $I+\mathcal M_{n,0}^{(4)}$, we have
\begin{equation} 
\label{equacaoj1maior}
X[x_a,x_b]x_c[x_d,x_e]
\equiv
X[x_c,x_b]x_a[x_d,x_e] -X[x_c,x_a]x_b[x_d,x_e],
\end{equation}
for every nonempty monomial $X$.
\end{lemma}

\begin{proof}
A direct expansion gives
$$
g=
\bigl([x_4,x_5][x_6,x_2]
-[x_2,x_5][x_6,x_4]\bigr)[x_3,x_1].
$$
Since the two summands are consequences of $q$ it follows that $g\in\langle q\rangle^T$.

Substitute $x_5\mapsto X$, $x_4\mapsto x_a$,
$x_2\mapsto x_b$, $x_6\mapsto x_c$,
$x_3\mapsto x_d$ and $x_1\mapsto x_e$
in the defining expression for $g$.
After orienting the commutators and ordering the initial
blocks, its first three summands belong to
$\mathcal M_{n,0}^{(4)}$ modulo $I$.
The remaining three summands give the required congruence.
\end{proof} 

In particular, for $s=1$, we may assume $j_1>t_1$ modulo
$I+\mathcal M_{n,0}^{(4)}$. Now, set
$$
S(a,b,c)
:=St_3(x_a,x_b,x_c)=
[x_a,x_b]x_c-[x_c,x_b]x_a+[x_c,x_a]x_b.$$

 We now extend (\ref{equacaoj1maior}) by allowing an arbitrary word between the first commutator and the distinguished variable $x_c$.

\begin{lemma} \label{lem:g-middle-word}
Let $X$ be a nonempty multilinear monomial and let
$u=x_{t_1}\cdots x_{t_m}$, where $m\geq0$.
Set $n=\deg X+m+5$ and $B=[x_d,x_e]$. Then
\[
X[x_a,x_b]ux_cB
\equiv X[x_c,x_b]ux_aB-X[x_c,x_a]ux_bB
\pmod{I+\sum_{r=0}^{m}\mathcal M_{n,r}^{(4)}}.
\]
\end{lemma}

\begin{proof}
For $m=0$, the result is Lemma~\ref{lem:g-consequence}.
Assume $m\geq1$ and write
$S(a,b,cu)=\operatorname{St}_3(x_a,x_b,x_cu)$.
The alternative expansion of the standard polynomial gives
\[
XS(a,b,cu)B
=Xx_cu[x_a,x_b]B-Xx_a[x_cu,x_b]B
+Xx_b[x_cu,x_a]B.
\]
For $z\in\{a,b\}$, use
\[
[x_cu,x_z]=[x_c,x_z]u+x_c[u,x_z],
\qquad
[u,x_z]=\sum_{h=1}^{m}
x_{t_1}\cdots x_{t_{h-1}}
[x_{t_h},x_z]x_{t_{h+1}}\cdots x_{t_m}.
\]
Every resulting term has at most $m$ variables between its
two simple commutators. Hence, after orienting the commutators
and ordering the monomial blocks,
\[
XS(a,b,cu)B\in I+\sum_{r=0}^{m}\mathcal M_{n,r}^{(4)}.
\]
On the other hand,
\begin{align*}
XS(a,b,cu)B
={}&X[x_a,x_b]x_cuB-X[x_c,x_b]ux_aB
+X[x_c,x_a]ux_bB-Xx_c[u,x_b]x_aB+Xx_c[u,x_a]x_bB.
\end{align*}
In each of the last two terms, the same expansion of $[u,x_z]$
leaves $m-h+1\leq m$ variables between the commutators.
Thus these terms also belong to
$I+\sum_{r=0}^{m}\mathcal M_{n,r}^{(4)}$. Therefore, 
$$X[x_a,x_b]x_cuB\equiv X[x_c,x_b]ux_aB
-X[x_c,x_a]ux_bB \pmod{I+\sum_{r=0}^{m}\mathcal M_{n,r}^{(4)}}.$$

Since $X[x_a,x_b](x_cu-ux_c)B\in I$, the result follows. 
\end{proof} \black 

We can now obtain the first order condition on the indices.

\begin{lemma}\label{lem:first-index-larger-middle}
The space $\mathcal M_{n,s}^{(4)}$, $1\leq s\leq n-5$,
is generated, modulo
$I+\sum_{k=0}^{s-1}\mathcal M_{n,k}^{(4)}$, by the polynomials
\[
x_{l_1}\cdots x_{l_r}
[x_{j_1},x_{i_1}]
x_{t_1}\cdots x_{t_s}
[x_{j_2},x_{i_2}]
\in\mathcal M_{n,s}^{(4)}
\]
under the condition $j_1>t_s$.
\end{lemma}

\begin{proof}
Let
\[
m=X[x_{j_1},x_{i_1}]
x_{t_1}\cdots x_{t_s}[x_{j_2},x_{i_2}],
\qquad t_1<\cdots<t_s.
\]
There is nothing to prove if $j_1>t_s$.
Suppose that $j_1<t_s$ and set $a=j_1$, $b=i_1$, $c=t_s$,
$u=x_{t_1}\cdots x_{t_{s-1}}$ and $B=[x_{j_2},x_{i_2}]$.
Since $c>a>b$, Lemma~\ref{lem:g-middle-word} gives
\[
m\equiv X[x_c,x_b]ux_aB-X[x_c,x_a]ux_bB
\pmod{I+\sum_{k=0}^{s-1}\mathcal M_{n,k}^{(4)}}.
\]
In both terms the first index of the first commutator is
$c=t_s$, while all indices in the new middle block are smaller
than $c$. After ordering that block, both terms satisfy the
required condition.
\end{proof}

We now prove the analogous reduction for the second commutator. Set $p=x_1[x_2,x_3][x_4,x_5]x_6$.

\begin{lemma}\label{lem:st3-consequence}
We have
$x_6[x_5,x_1]\operatorname{St}_3(x_4,x_3,x_2)
\in\langle p\rangle^T$. More precisely,
\begin{align*}
x_6[x_5,x_1]\operatorname{St}_3(x_4,x_3,x_2)
={}&x_6[x_5,x_1][x_3,x_2]x_4
-x_6[x_5,x_1][x_4,x_2]x_3+x_6[x_5,x_1][x_4,x_3]x_2.
\end{align*}
\end{lemma}

\begin{proof}
Recall that
$\operatorname{St}_3(x_4,x_3,x_2)
=x_4[x_3,x_2]-x_3[x_4,x_2]+x_2[x_4,x_3]$.
Set $A=x_6[x_5,x_1]$.
Using $Ax[y,z]=A[x,[y,z]]+A[y,z]x$, we obtain
\begin{align*}
A\operatorname{St}_3(x_4,x_3,x_2)
={}&A\bigl([x_4,[x_3,x_2]]-[x_3,[x_4,x_2]]
+[x_2,[x_4,x_3]]\bigr)+A[x_3,x_2]x_4-A[x_4,x_2]x_3+A[x_4,x_3]x_2.
\end{align*}
The expression in parentheses vanishes by the Jacobi identity.
Each remaining summand is a consequence of $p$,
which proves the result.

\end{proof}

For distinct indices $c,d,e$, set
\[
S'(c,d,e):=\operatorname{St}_3(x_c,x_d,x_e)
=x_c[x_d,x_e]-x_d[x_c,x_e]+x_e[x_c,x_d].
\]
The previous lemma gives $X[x_a,x_b]S'(c,d,e)\in I$
for every nonempty monomial $X$.
We now allow an arbitrary word between the first commutator
and the variables occurring in $S'(c,d,e)$.

\begin{lemma}\label{lem:st3-middle-word}

Let $X$ be a nonempty monomial and let
$u=x_{t_1}\cdots x_{t_m}$, where $m\geq0$.
Set $n=\deg X+m+5$.
Then, modulo $I+\sum_{r=0}^{m}\mathcal M_{n,r}^{(4)}$, we have
\begin{equation}\label{equacaoj2maior}
X[x_a,x_b]ux_c[x_d,x_e]
\equiv X[x_a,x_b]ux_d[x_c,x_e]
-X[x_a,x_b]ux_e[x_c,x_d].
\end{equation}
\end{lemma} 

\begin{proof}
For $m=0$, the result follows directly from
Lemma~\ref{lem:st3-consequence}.
Assume $m\geq1$.
Since $I$ is a $T$-ideal, substitute $x_6\mapsto X$,
$x_5\mapsto x_au$, $x_1\mapsto x_b$, $x_4\mapsto x_c$,
$x_3\mapsto x_d$ and $x_2\mapsto x_e$ in that lemma.
Thus $X[x_au,x_b]S'(c,d,e)\in I$.
Using $[x_au,x_b]=x_a[u,x_b]+[x_a,x_b]u$, we obtain
\[
X[x_a,x_b]uS'(c,d,e)
\equiv-Xx_a[u,x_b]S'(c,d,e)\pmod I.
\]
Now
\[
[u,x_b]=\sum_{h=1}^{m}
x_{t_1}\cdots x_{t_{h-1}}
[x_{t_h},x_b]x_{t_{h+1}}\cdots x_{t_m}.
\]
After expanding $S'(c,d,e)$, every summand on the right has
$m-h+1\leq m$ variables between its two simple commutators.
Orienting the commutators and ordering the monomial blocks
therefore gives
\[
X[x_a,x_b]uS'(c,d,e)
\equiv0\pmod{I+\sum_{r=0}^{m}\mathcal M_{n,r}^{(4)}}.
\]
Expanding $S'(c,d,e)$ gives \eqref{equacaoj2maior}.

\end{proof}

We can now impose the second order condition on the middle block. 

\begin{lemma}\label{lem:second-index-larger-middle}
The space $\mathcal M_{n,s}^{(4)}$, $1\leq s\leq n-5$,
is generated, modulo
$I+\sum_{k=0}^{s-1}\mathcal M_{n,k}^{(4)}$, by the polynomials
\[
x_{l_1}\cdots x_{l_r}
[x_{j_1},x_{i_1}]
x_{t_1}\cdots x_{t_s}
[x_{j_2},x_{i_2}]
\]
satisfying $j_1>t_s$ and $j_2>t_s$.
\end{lemma}

\begin{proof}
Let $m=X[x_{j_1},x_{i_1}]
x_{t_1}\cdots x_{t_s}[x_{j_2},x_{i_2}]$,
where $t_1<\cdots<t_s$.
By Lemma~\ref{lem:first-index-larger-middle}, we may assume $j_1>t_s$.
There is nothing to prove if $j_2>t_s$.
Otherwise, $t_s>j_2>i_2$.
Set $u=x_{t_1}\cdots x_{t_{s-1}}$.
By Lemma~\ref{lem:st3-middle-word},
\[
m\equiv
X[x_{j_1},x_{i_1}]ux_{j_2}[x_{t_s},x_{i_2}]
-X[x_{j_1},x_{i_1}]ux_{i_2}[x_{t_s},x_{j_2}]
\pmod{I+\sum_{k=0}^{s-1}\mathcal M_{n,k}^{(4)}}.
\]
In both terms the first index of the second commutator is the
original $t_s$ and every index in the new middle block is
smaller than it. The condition on the first commutator is also
preserved, since $j_1$ is larger than the original $t_s$.
After ordering the middle blocks, both terms satisfy the
required conditions.
\end{proof} 

It remains to obtain the last order condition.
For distinct indices $a,b,c,d,e,l$, define
\begin{align*}
h_l(a,b,c,d,e)={}&
x_l[x_a,x_b]x_c[x_d,x_e]
-x_ax_l[x_b,x_c][x_d,x_e]
+x_dx_l[x_b,x_c][x_a,x_e]\\
&-x_ex_l[x_b,x_c][x_a,x_d]
+x_bx_l[x_e,x_c][x_a,x_d]
-x_bx_l[x_d,x_c][x_a,x_e]\\
&+x_bx_l[x_a,x_c][x_d,x_e]
-x_cx_l[x_e,x_b][x_a,x_d]
+x_cx_l[x_d,x_b][x_a,x_e]\\
&-x_cx_l[x_a,x_b][x_d,x_e]
-x_l[x_e,x_c]x_b[x_a,x_d]
+x_l[x_d,x_c]x_b[x_a,x_e]\\
&-x_l[x_a,x_c]x_b[x_d,x_e]
+x_l[x_e,x_b]x_c[x_a,x_d]
-x_l[x_d,x_b]x_c[x_a,x_e].
\end{align*}

\begin{lemma}\label{lem:garnir}
The polynomial $h_l(a,b,c,d,e)$ belongs to $\langle p,q\rangle_T\subseteq I.$
\end{lemma}

\begin{proof}
A direct expansion gives
\begin{align*}
h_l(a,b,c,d,e)={}&
-[x_c,x_l][x_b,x_e][x_d,x_a]
-[x_c,x_l][x_b,x_d][x_a,x_e]-[x_c,x_l][x_b,x_a][x_e,x_d]\\
&
+[x_b,x_l][x_c,x_e][x_d,x_a]+[x_b,x_l][x_c,x_d][x_a,x_e]
-[x_b,x_l][x_c,x_a][x_d,x_e]\\
&+[x_e,x_l][x_c,x_b][x_a,x_d]
+[x_d,x_l][x_c,x_b][x_e,x_a]+[x_a,x_l][x_c,x_b][x_d,x_e]\\
&
-x_l[x_c,x_b][x_e,x_d]x_a+x_l[x_c,x_b][x_e,x_a]x_d
-x_l[x_c,x_b][x_d,x_a]x_e.
\end{align*}
The first nine summands are consequence of $q=[x_1,x_2][x_3,x_4][x_5,x_6],$ whereas the last three summands are consequence of $p=x_1[x_2,x_3][x_4,x_5]x_6.$ Hence $h_l(a,b,c,d,e)\in\langle p,q\rangle_T.$
\end{proof}

We now use this relation for an arbitrary nonempty middle block.

\begin{lemma}\label{lem:garnir-middle-word}
Let $1\leq s\leq n-5$,
$u=x_{t_2}\cdots x_{t_s}$, with $t_1<\cdots<t_s$,
and let $X$ be a nonempty monomial.
Assume that $a>b$, $d>e$, $a>t_s$, $d>t_s$,
$c=t_1<b$ and $d<a$.
Then $X[x_a,x_b]x_cu[x_d,x_e]$ is congruent modulo $I$
to a linear combination of polynomials of the following types:
\begin{enumerate}
\item polynomials belonging to
$\sum_{r=0}^{s-1}\mathcal M_{n,r}^{(4)}$;
\item polynomials in $\mathcal M_{n,s}^{(4)}$
satisfying $i_1<t_1$ or $j_1<j_2$;
\item polynomials in $\mathcal M_{n,s}^{(4)}$
having a first middle index strictly smaller than $c$.
\end{enumerate}
In \textup{(2)} and \textup{(3)}, the polynomials also satisfy
$j_1>t_s$ and $j_2>t_s$, with respect to their new indices.
\end{lemma}

\begin{proof}
Write $X=Yx_l$, where $Y$ may be empty.
By Lemma~\ref{lem:garnir}, substitute $x_c\mapsto x_cu$
in $h_l(a,b,c,d,e)$ and multiply on the left by $Y$.
Thus $Yh_l(a,b,cu,d,e)\in I$ and its first summand is
$X[x_a,x_b]x_cu[x_d,x_e]$.

For $s\geq2$, we analyze the remaining terms using
\[
[x_z,x_cu]=[x_z,x_c]u+x_c[x_z,u],
\qquad
[x_z,u]=\sum_{h=2}^{s}
x_{t_2}\cdots x_{t_{h-1}}
[x_z,x_{t_h}]x_{t_{h+1}}\cdots x_{t_s}.
\]
If $x_cu$ precedes both commutators, the term belongs to
$\mathcal M_{n,0}^{(4)}$ modulo $I$.
If $[x_z,x_cu]$ is immediately followed by the second
commutator, its expansion leaves at most $s-1$ middle
variables.
For a term $X[x_z,x_cu]x_bC$, where $C$ is the second
commutator, the part $X[x_z,x_c]ux_bC$ retains $s$ middle
variables, whereas every term from $Xx_c[x_z,u]x_bC$
has $s-h+1\leq s-1$ middle variables.
For $s=1$, the same grouping follows directly by omitting $u$.
Consequently, after ordering the monomial blocks,
\begin{align*}
X[x_a,x_b]x_cu[x_d,x_e]\equiv{}&
 X[x_e,x_c]ux_b[x_a,x_d]
-X[x_d,x_c]ux_b[x_a,x_e]+X[x_a,x_c]ux_b[x_d,x_e]\\
&
-X[x_e,x_b]x_cu[x_a,x_d]+X[x_d,x_b]x_cu[x_a,x_e]
\pmod{I+\sum_{r=0}^{s-1}\mathcal M_{n,r}^{(4)}}.
\end{align*}

In the first, second, fourth and fifth terms, the first
index of the second commutator is $a$, which is larger
than every other index in the two commutators and the
middle block. After orienting the first commutator, applying Lemma~\ref{lem:first-index-larger-middle}
when necessary preserves this property. Thus all resulting
terms of middle length $s$ satisfy both upper-index
conditions and $j_1<j_2$.

It remains to consider
$X[x_a,x_c]ux_b[x_d,x_e]$.
All indices in $ux_b$ are larger than $c$, so this term
satisfies $i_1<t_1$.
If $d>b$, it already satisfies both upper-index conditions.
If $b>d$, then $b$ is larger than every index in $u$.
Lemma~\ref{lem:st3-middle-word} gives
\begin{align*}
X[x_a,x_c]ux_b[x_d,x_e]\equiv 
X[x_a,x_c]ux_d[x_b,x_e]
-X[x_a,x_c]ux_e[x_b,x_d]
\pmod{I+\sum_{r=0}^{s-1}\mathcal M_{n,r}^{(4)}}.
\end{align*}
Both terms satisfy the upper-index conditions.
The first satisfies $i_1<t_1$, since every index in $ux_d$
is larger than $c$.
The second also satisfies this inequality if $e>c$;
if $e<c$, its first middle index is $e<c$.
This proves all the assertions, including the case $s=1$,
when $u$ is empty.

\end{proof} 

We can now complete the normalization of the indices.

\begin{lemma}\label{lem:last-order-condition}
The space $\mathcal M_{n,s}^{(4)}$, $1\leq s\leq n-5$,
is generated, modulo
$I+\sum_{k=0}^{s-1}\mathcal M_{n,k}^{(4)}$, by polynomials
\[
x_{l_1}\cdots x_{l_r}
[x_{j_1},x_{i_1}]
x_{t_1}\cdots x_{t_s}
[x_{j_2},x_{i_2}]
\]
satisfying $j_1>t_s$, $j_2>t_s$ and either
$i_1<t_1$ or $j_1<j_2$.
\end{lemma}

\begin{proof}
Fix $s\geq1$. By Lemmas~\ref{lem:first-index-larger-middle}
and~\ref{lem:second-index-larger-middle}, we may assume $j_1>t_s$
and $j_2>t_s$.
We argue by induction on the first middle index $t_1$.
Let
\[
m=X[x_{j_1},x_{i_1}]
x_{t_1}\cdots x_{t_s}[x_{j_2},x_{i_2}],
\qquad t_1<\cdots<t_s.
\]
There is nothing to prove if $i_1<t_1$ or $j_1<j_2$.
Otherwise, $t_1<i_1$ and $j_2<j_1$.
Apply Lemma~\ref{lem:garnir-middle-word} with
$a=j_1$, $b=i_1$, $c=t_1$, $d=j_2$, $e=i_2$
and $u=x_{t_2}\cdots x_{t_s}$.
Modulo $I+\sum_{k=0}^{s-1}\mathcal M_{n,k}^{(4)}$,
the resulting terms either satisfy the required alternative
or have a strictly smaller first middle index.
In both cases the upper-index conditions are preserved.

If $t_1=1$, the second possibility cannot occur.
If $t_1>1$, it is covered by the induction hypothesis.
Therefore every term reduces to a linear combination of
polynomials with the required conditions.

\end{proof}


\subsection{Linear independence} We now collect the normal forms obtained in the previous reductions. For $n\geq5$, set
$\mathcal B_{n,0}^{(4)}=\mathcal B'\cup\mathcal B''$,
where $\mathcal B'$ and $\mathcal B''$ are the families
defined in Subsection~3.1.

For $1\leq s\leq n-5$, let $\mathcal B_{n,s}^{(4)}$ be the set of all polynomials
$x_{l_1}\cdots x_{l_r}
[x_{j_1},x_{i_1}]
x_{t_1}\cdots x_{t_s}
[x_{j_2},x_{i_2}],$ where $r=n-s-4\geq1,$ $
l_1<\cdots<l_r,$ $
t_1<\cdots<t_s,$ $
j_1>i_1,$ $j_2>i_2,$ $j_1>t_s,$ $
j_2>t_s$ and either $i_1<t_1$ or $j_1<j_2.$ All indices are distinct and together form $\{1,\ldots,n\}$. Finally, set
$$
\mathcal B_n^{(4)}
=
\mathcal B_{n,0}^{(4)}
\cup
\bigcup_{s=1}^{n-5}\mathcal B_{n,s}^{(4)}.
$$
By the previous lemmas, $\mathcal B_n^{(4)}$ generates the polynomials of type $(4')$ modulo $I$.

The evaluations used in the proof of Lemma~\ref{lemtype123} also annihilate every polynomial of type $(4')$. Hence it remains to prove
that $\mathcal B_n^{(4)}$ is linearly independent modulo
$P_n\cap\operatorname{Id}^z(D_0)$.

\begin{lemma}
\label{lem:linear-independence-B4}
The polynomials in $\mathcal B_n^{(4)}$ are linearly independent
modulo $P_n\cap\operatorname{Id}^{z}(D_0)$.
\end{lemma}

\begin{proof}
Suppose that $f=\sum_{m\in\mathcal B_n^{(4)}}\alpha_m m
 \in \operatorname{Id}^{z}(D_0).$ We prove that every coefficient $\alpha_m$ is zero.

For $m=
 x_{l_1}\cdots x_{l_r}
 [x_{j_1},x_{i_1}]
 x_{t_1}\cdots x_{t_s}
 [x_{j_2},x_{i_2}]
 \in\mathcal B_{n,s}^{(4)},$ write $L=\{l_1,\ldots,l_r\}$, $T=\{t_1,\ldots,t_s\},$ and set $\kappa(m)=(j_2,j_1,-i_1,-i_2),$ where these tuples are ordered lexicographically.

Assume, by contradiction, that some coefficient in the above
linear combination is nonzero. Choose $s$ maximal such that
$\alpha_m\neq0$ for some
$m\in\mathcal B_{n,s}^{(4)}$ and, among all such polynomials,
choose $m$ with $\kappa(m)$ minimal.

We associate with $m$ an evaluation
$\varphi_m:F\langle X\rangle\rightarrow D_0$ as follows.

Suppose first that either $s=0$ and $m\in\mathcal B'$, or
$s\geq1$ and $i_1<t_1$. Define
\[
 \varphi_m(x_k)=
 \begin{cases}
  e_{22},& k\in L,\\
  e_{33},& k\in T\cup\{i_1,j_2\},\\
  e_{23},& k=j_1,\\
  e_{34},& k=i_2.
 \end{cases}
\] Then $\varphi_m(m)
 =
 e_{22}[e_{23},e_{33}]e_{33}^{\,s}[e_{33},e_{34}]
 =
 e_{24}.$

If $s=0$ and $m\in\mathcal B''$, so that $j_1=n$, define
\[
 \varphi_m(x_k)=
 \begin{cases}
  e_{22},& k\in L\cup\{j_1\},\\
  e_{23},& k=i_1,\\
  e_{33},& k=j_2,\\
  e_{34},& k=i_2.
 \end{cases}
\]
In this case, $\varphi_m(m)
 =
 e_{22}[e_{22},e_{23}][e_{33},e_{34}]
 =
 e_{24}.$

Finally, suppose that $s\geq1$ and $t_1<i_1$. By the definition
of $\mathcal B_{n,s}^{(4)}$, we then have $j_1<j_2$. Define
\[
 \varphi_m(x_k)=
 \begin{cases}
  e_{22},& k\in L,\\
  e_{33},& k\in T\cup\{j_1,j_2\},\\
  e_{23},& k=i_1,\\
  e_{34},& k=i_2.
 \end{cases}
\]
Now $\varphi_m(m)
 =
 e_{22}[e_{33},e_{23}]e_{33}^{\,s}[e_{33},e_{34}]
 =
 -e_{24}.$

We claim that, in all three cases, if
$m'\in\mathcal B_{n,s'}^{(4)}$
 and $
 \varphi_m(m')\neq0,$ then either
$s'=s+1,$ or
$s'=s$ and $\bigl(m'=m\ \text{or}\ \kappa(m')<\kappa(m)\bigr).$

Indeed, this follows directly by tracing the only possible nonzero
product of matrix units. The variables in the initial block of $m'$
must be evaluated at $e_{22}$; its first commutator must evaluate to
$\pm e_{23}$; all variables in its middle block must be evaluated at
$e_{33}$; and its second commutator must evaluate to $\pm e_{34}$.
Thus every nonzero value is equal to $\pm e_{24}$.

Suppose first that $m\notin\mathcal B''$. If the diagonal entry
paired with $e_{23}$ in the first commutator is an $e_{22}$, one
variable is removed from the initial block and one additional
$e_{33}$ occurs in the middle block. Consequently, $s'=s+1$.
Otherwise, the diagonal entry paired with $e_{23}$ is an $e_{33}$,
the initial block is unchanged and hence $s'=s$.

It remains to consider the latter possibility. In the first case
above, $j_2$ is the largest index among the variables evaluated at
$e_{33}$ that may occur in the middle block or in the second
commutator. The condition that the first entry of the second
commutator be larger than every middle index implies that either
$j_2$ remains in the second commutator or it is transferred to the
first commutator. In the second alternative, the first coordinate
of $\kappa(m')$ is strictly smaller than that of $\kappa(m)$. In the
first alternative, either the first commutator is the one occurring
in $m$, or its smaller entry is strictly larger than $i_1$; in the
latter case,
$-i_1'<-i_1,$ and therefore again $\kappa(m')<\kappa(m)$.

In the third case, we have
$j_2>j_1>t_s.$ The normal-form inequalities force $j_2$ to remain in the second
commutator, unless the first coordinate of $\kappa(m')$ decreases.
Once $j_2$ remains fixed, the same inequalities force $j_1$ to
remain in the first commutator. Hence either $m'=m$ or
$\kappa(m')<\kappa(m)$.

For $m\in\mathcal B''$, the variable $x_{j_1}=x_n$ is evaluated
at $e_{22}$, while $x_{j_2}$ and $x_{i_2}$ are the unique variables
evaluated at $e_{33}$ and $e_{34}$, respectively. Hence the second
commutator remains unchanged and $s'=0=s$. Any different nonzero
normal polynomial must replace $x_{j_1}$ in the first commutator
by some $x_l$ with $l\in L$. Its first upper index is therefore
strictly smaller than $n=j_1$ and so
$\kappa(m')<\kappa(m)$.

By the choice of $s$ and $\kappa(m)$, every other polynomial
with a nonzero value under $\varphi_m$ has zero coefficient
in $f$. Hence $\varphi_m(f)=\alpha_m\varphi_m(m)=\pm\alpha_m e_{24}$.
Since $f$ is central and $e_{24}\notin Z(D_0)=Fe_{14}$,
we obtain $\alpha_m=0$, a contradiction.
This proves the linear independence of $\mathcal B_n^{(4)}$
modulo $P_n\cap\operatorname{Id}^z(D_0)$.
\end{proof}

Combining Lemma~\ref{lem:linear-independence-B4} with Lemma~\ref{lemtype123}, we obtain that
$x_1\cdots x_n$ and all the normal forms constructed above
are linearly independent modulo
$P_n\cap\operatorname{Id}^z(D_0)$.

\subsection{The main result} For $n\geq5$, let $\mathcal B_n$ consist of the polynomial $x_1\cdots x_n$, the polynomials of types $(1)$, $(2)$ and $(3)$
defined above and the elements of $\mathcal B_n^{(4)}$
introduced in the preceding subsection.

\begin{theorem}\label{thm:central-polynomials-D0}
For the algebra $D_0$ the following statements hold.

\begin{enumerate}
\item[(1)]
The space of central polynomials of $D_0$ is the $T$-ideal
$$
\operatorname{Id}^z(D_0)
=
\left\langle
x_1[x_2,x_3][x_4,x_5]x_6,\,
[x_1,x_2][x_3,x_4][x_5,x_6],\,
St_5(x_1,\ldots,x_5)
\right\rangle_T.
$$

\item[(2)]
For every $n\geq5$, the family $\mathcal B_n$ is a basis of
${P_n}$ modulo ${P_n\cap\operatorname{Id}^z(D_0)}.$

\item[(3)]
For every $n\geq5$,
$
c_n^z(D_0)
=
2^{n-1}(n^2-4n+2)-n^2+7n-3$ and $\delta_n(D_0)
=
2^{n-3}(n^3-9n^2+20n-8)
+
3n^2-9n+3.$
\end{enumerate}
\end{theorem}

\begin{proof} We have already proved that
$I\subseteq\operatorname{Id}^z(D_0).$ The reduction lemmas of the previous subsections show that every polynomial in $P_n$ is congruent modulo $I$ to a linear combination of the elements of $\mathcal B_n$. By the linear independence proved above, no nonzero linear combination of elements of $\mathcal B_n$ is a central polynomial of $D_0$. Therefore $P_n\cap\operatorname{Id}^z(D_0)
=P_n\cap I$
for every $n\geq5$. Since $D_0$ has no nonzero central polynomials of degree at most four and the ground field has characteristic zero,
multilinearization yields $\operatorname{Id}^z(D_0)=I$. This proves items (1) and (2).

It remains to count the elements of $\mathcal B_n$. The polynomial $x_1\cdots x_n$ together with the first two
families contributes $c_n(UT_2)=2^{n-1}(n-2)+2.$ Denoting the family of type $(3)$ by $\mathcal B_n^{(3)}$,
we have
\begin{align*}
\left|\mathcal B_n^{(3)}\right|=
\sum_{k=2}^{n-2}
\binom{n}{k}
(k-1)(n-k-1)=
(n^2-5n+4)2^{n-2}+2(n-1).
\end{align*}

For the fourth family, we already know that
$$
\left|\mathcal B_{n,0}^{(4)}\right|
=
6\binom{n}{4}
-
\binom{n-1}{4}.
$$
Now fix $1\leq s\leq n-5$. Choose the $s+4$ indices which occur in the middle block and in the two commutators. There are $\binom{n}{s+4}
$ choices. Once these indices are fixed and arranged increasingly, the condition
$
j_1>t_s$ and $j_2>t_s$
implies that the position of $t_s$ can only be $s$, $s+1$ or $s+2$. The numbers of admissible arrangements in these three cases are, respectively,
$3$, $3s+1$ and $(s+1)^2$. Hence
$$
\left|\mathcal B_{n,s}^{(4)}\right|
=
(s^2+5s+5)
\binom{n}{s+4}.
$$
Consequently,
$$\left|\mathcal B_n^{(4)}\right|
=
6\binom{n}{4}
-
\binom{n-1}{4}
+
\sum_{s=1}^{n-5}
(s^2+5s+5)
\binom{n}{s+4}=
(n^2-5n+4)2^{n-2}
-
n^2+5n-3.$$

Therefore
\begin{align*}
c_n^z(D_0)
={}&
c_n(UT_2)
+
\left|\mathcal B_n^{(3)}\right|
+
\left|\mathcal B_n^{(4)}\right|
=
2^{n-1}(n^2-4n+2)-n^2+7n-3.
\end{align*}

For $n\geq7$, the formula for $\delta_n(D_0)$ follows from
Theorem~\ref{identidadesordinarias} and
$\delta_n(D_0)=c_n(D_0)-c_n^z(D_0)$. The cases $n\in \{5,6\}$ can be verified directly using the basis described above and Theorem \ref{identidadesordinarias}\textup{(1)}.
\end{proof}

\section{Group gradings on varieties of almost polynomial $\delta$-growth}

Throughout this section, $G$ denotes an arbitrary group with identity
$1_G$, unless otherwise stated.

Group gradings on the Grassmann algebra have been extensively studied;
for homogeneous gradings and their graded identities, see
\cite{DiVincenzoDaSilva2009,Centrone2016,GuimaraesKoshlukov2021},
while nonhomogeneous $\mathbb Z_2$-gradings were considered in
\cite{FidelesEtAl2024,KoshlukovGuimaraes2023}.

Given $\mathbf g=(g_1,\ldots,g_m)\in G^m$, the associated
\emph{elementary grading} on $UT_m$ is defined by
\[
\deg e_{ij}=g_i^{-1}g_j,\qquad 1\leq i\leq j\leq m.
\]
We also call elementary the grading induced on any graded subalgebra
of $UT_m$. Multiplying all entries of $\mathbf g$ on the left by the
same element does not change the grading, so we may assume $g_1=1_G$.

For the algebra $D$, Gomez Parada and Koshlukov proved that every
grading by an abelian group is isomorphic to an elementary grading
\cite[Theorem~1]{Gomezparada}. We now establish the corresponding
result for $D_0$, with no restriction on the grading group.

\begin{lemma}\label{gradedanulador}
Let $(D_0,\Gamma)$ be a $G$-graded algebra and let $S\subseteq D_0$
be a graded subspace. Then its right, left and two-sided annihilators
$\operatorname{Ann}_{D_0}^{r}(S),$ $
\operatorname{Ann}_{D_0}^{l}(S),$ $
\operatorname{Ann}_{D_0}(S)$ are graded subspaces of $D_0$.
\end{lemma}
\begin{proof}
Let $x=\sum_{g\in G}x_g\in\operatorname{Ann}_{D_0}^{r}(S)$,
where $x_g\in(D_0)_g$. For $s\in S\cap(D_0)_h$, we have
\[
0=sx=\sum_{g\in G}sx_g.
\]
Since the degrees $hg$ are distinct as $g$ varies, $sx_g=0$ for
every $g$. As $S$ is graded, each $x_g$ belongs to
$\operatorname{Ann}_{D_0}^{r}(S)$. The left annihilator is treated
similarly and the two-sided annihilator is their intersection.
\end{proof}

\begin{lemma} \label{idempotents}
Every $G$-grading $\Gamma$ on $D_0$ is isomorphic to a grading
$\Gamma'$ in which $e_{22}$ and $e_{33}$ are homogeneous of degree $1_G$.
\end{lemma}

\begin{proof}Let $J=J(D_0)$. Since $D_0$ is finite-dimensional and
$\operatorname{char}F=0$, the radical $J$ is graded. Moreover,
$D_0/J\cong F\oplus F$ is separable, so the graded
Wedderburn--Malcev theorem yields a graded semisimple subalgebra
$B\subseteq D_0$ such that $D_0=B\oplus J$;
see \cite[Section~3]{GordienkoCostability}.

Set $Q=Fe_{22}\oplus Fe_{33}$. By the conjugacy part of the ordinary
Wedderburn--Malcev theorem, applied to the unitization of $D_0$,
there is an automorphism $\psi$ of $D_0$ such that $\psi(B)=Q$.
Define the transported grading $\Gamma'$ by
$(D_0)'_g=\psi((D_0)_g)$ for $g\in G$. Then $\Gamma'$ is isomorphic
to $\Gamma$ and both $Q$ and $J$ are graded with respect to $\Gamma'$.

Since $E=e_{22}+e_{33}$ is the identity of the graded algebra $Q$,
it has degree $1_G$. Hence $S=EJE=Fe_{23}$ is graded.
By Lemma~\ref{gradedanulador}, the subspaces
\[
Q\cap\operatorname{Ann}_{D_0}^{r}(S)=Fe_{22},
\qquad
Q\cap\operatorname{Ann}_{D_0}^{l}(S)=Fe_{33}
\]
are graded. Thus $e_{22}$ and $e_{33}$ are homogeneous and their
idempotency forces both degrees to be $1_G$.
\end{proof}

\begin{lemma} \label{elementoshomogeneos}
A $G$-grading $\Gamma$ on $D_0$ is elementary if and only if
$e_{22},e_{33},e_{12},e_{23},e_{34}$ are homogeneous.
\end{lemma}

\begin{proof}
The forward implication follows from the definition.
Conversely, suppose that these matrix units are homogeneous and write
$\deg e_{12}=g$, $\deg e_{23}=r$ and $\deg e_{34}=h.$ The idempotents $e_{22}$ and $e_{33}$ have degree $1_G$. Since
\[
e_{13}=e_{12}e_{23},\qquad
e_{24}=e_{23}e_{34},\qquad
e_{14}=e_{12}e_{23}e_{34},
\]
all basis matrix units are homogeneous with the degrees prescribed
by the tuple $(1_G,g,gr,grh)$. Hence $\Gamma$ is elementary.
\end{proof}

\begin{lemma} \label{e22e33}
A $G$-grading $\Gamma$ on $D_0$ is elementary if and only if
$e_{22},e_{33}\in(D_0)_{1_G}$.
\end{lemma}

\begin{proof}
The forward implication is immediate. Conversely, suppose that
$e_{22},e_{33}\in(D_0)_{1_G}$ and set $J=J(D_0)$. Since $J$ is graded,
the one-dimensional subspaces
\[
Je_{22}=Fe_{12},\qquad
e_{22}Je_{33}=Fe_{23},\qquad
e_{33}J=Fe_{34}
\]
are graded. Thus $e_{12},e_{23},e_{34}$ are homogeneous and
Lemma~\ref{elementoshomogeneos} completes the proof.
\end{proof}

Combining Lemmas~\ref{idempotents} and
\ref{e22e33}, we obtain the main result of this section.

\begin{theorem} \label{gradingond0}
Let $G$ be an arbitrary group and let $F$ be a field of characteristic
zero. Then every $G$-grading on $D_0$ is isomorphic to an elementary
grading.
\end{theorem}

For $G=\mathbb Z_2$, we use additive notation. The matrix units
$e_{ij}$ with $1\leq i<j\leq4$ form a basis of the Jacobson
radical $J(D_0)$ and will be called \emph{radical matrix units}.
A homogeneous matrix unit is called \emph{even} or \emph{odd}
if it has degree $0$ or $1$, respectively.

The eight normalized tuples $(0,g_2,g_3,g_4)$ give the elementary
gradings listed in Table~\ref{tab:superalgebras-D0}.
The diagonal matrix units $e_{22}$ and $e_{33}$ are even in
every case, so the table lists only the radical matrix units,
separated according to their parity.

\begingroup
\renewcommand{\arraystretch}{1.15}
\setlength{\LTcapwidth}{\textwidth}

\begin{longtable}[c]{@{}ccll@{}}

    \caption{The elementary superalgebra structures on $D_0$.}
    \label{tab:superalgebras-D0}\\

    \toprule
    Algebra
    & Defining tuple
    & Even radical units
    & Odd radical units \\
    \midrule
    \endfirsthead

    \bottomrule
    \endlastfoot

    $\mathcal{A}_1$
    & $(0,0,0,0)$
    & $e_{12}, e_{13}, e_{14}, e_{23}, e_{24}, e_{34}$
    & $-$ \\

    $\mathcal{A}_2$
    & $(0,0,0,1)$
    & $e_{12}, e_{13}, e_{23}$
    & $e_{14}, e_{24}, e_{34}$ \\

    $\mathcal{A}_3$
    & $(0,0,1,1)$
    & $e_{12}, e_{34}$
    & $e_{13}, e_{14}, e_{23}, e_{24}$ \\

    $\mathcal{A}_4$
    & $(0,0,1,0)$
    & $e_{12}, e_{14}, e_{24}$
    & $e_{13}, e_{23}, e_{34}$ \\

    $\mathcal{A}_5$
    & $(0,1,1,1)$
    & $e_{23}, e_{24}, e_{34}$
    & $e_{12}, e_{13}, e_{14}$ \\

    $\mathcal{A}_6$
    & $(0,1,1,0)$
    & $e_{14}, e_{23}$
    & $e_{12}, e_{13}, e_{24}, e_{34}$ \\

    $\mathcal{A}_7$
    & $(0,1,0,0)$
    & $e_{13}, e_{14}, e_{34}$
    & $e_{12}, e_{23}, e_{24}$ \\

    $\mathcal{A}_8$
    & $(0,1,0,1)$
    & $e_{13}, e_{24}$
    & $e_{12}, e_{14}, e_{23}, e_{34}$ \\

\end{longtable}
\endgroup

In the following sections, we determine the graded identities
and codimensions of these superalgebras and show that they are
pairwise non-PI-equivalent as $\mathbb Z_2$-graded algebras.

\section{Graded identities of $\mathcal A_1$, $\mathcal A_2$
and $\mathcal A_5$}

In this section, we determine generators for the graded identities of $\mathcal A_1$, $\mathcal A_2$ and $\mathcal A_5$, describe multilinear bases for their relatively free graded algebras and compute their graded codimensions.

For $n\geq1$ and $0\leq q\leq n$, let $P_{n-q,q}$ denote the
subspace of $P_n^{\mathbb Z_2}$ spanned by multilinear monomials
containing exactly $q$ odd variables, allowing all choices of
their labels among $\{1,\ldots,n\}$. Set
\[
c_{n-q,q}(A)
=
\dim_F
\frac{P_{n-q,q}}
{P_{n-q,q}\cap\textnormal{Id}^{\mathbb Z_2}(A)}.
\]
Thus the choices of the odd labels are already included in this
number and
\begin{equation} \label{eqcodimensoes}
c_n^{\mathbb Z_2}(A)
=
\sum_{q=0}^{n}c_{n-q,q}(A).    
\end{equation}

We first consider the trivial grading $\mathcal A_1$. The following
result is an immediate consequence of Theorem~\ref{identidadesordinarias}. 
 
 \begin{theorem}  For the algebra $\mathcal{A}_1$ we have:
    \begin{enumerate}
        \item[(1)] $\textnormal{Id}^{\mathbb{Z}_2}(\mathcal{A}_1)= \langle y_1[y_2,y_3][y_4,y_5]y_6, z_1 \rangle_{T_{\mathbb{Z}_2}}$.
        \item[(2)] $c_n^{\mathbb{Z}_2}(\mathcal{A}_1)=2^{n-3}n(n-1)(n-4)+2n(n-1)$, for $n\geq 7$. 
    \end{enumerate}
\end{theorem}

 We now turn to $\mathcal A_2$. Recall that
\[
\begin{aligned}
(\mathcal A_2)_{0}
=F e_{22}+F e_{33}+F e_{12}+F e_{13}+F e_{23}\quad \mbox{ and }\quad 
(\mathcal A_2)_{1}
=F e_{14}+F e_{24}+F e_{34}.
\end{aligned}
\]
Thus, the polynomials
\begin{equation}\label{sec6:a2-identities}
y_1[y_2,y_3][y_4,y_5],\qquad
z_1z_2,\qquad
z_1y_2
\end{equation}
are graded identities of $\mathcal A_2$. 

Let $I$ be the ${T_{\mathbb Z_2}}$-ideal generated by the polynomials in (\ref{sec6:a2-identities}). Then $I\subseteq\textnormal{Id}^{\mathbb Z_2}(\mathcal A_2).$  As a consequence,
$P_{n-q,q}\subseteq I$ for all  $2\leq q\leq n.$ Therefore, it suffices to consider the cases $q\in\{0,1\}$.

\begin{itemize}
    \item {\bf Case $q=0$.} \end{itemize}

By Proposition~\ref{prop:UT2-central}, every polynomial $f\in\operatorname{Id}(UT_2)$ in the even variables satisfies $y_t f\in I$. Indeed, multiplying a consequence of $[y_1,y_2][y_3,y_4]$ on the left by $y_t$ gives a consequence of $y_1[y_2,y_3][y_4,y_5]$. Thus, keeping the first variable fixed and applying the standard multilinear reduction for $UT_2$ to the remaining word, we obtain the following spanning
family modulo $I$:
\begin{equation}\label{sec6:a2-even}
y_t y_{i_1}\cdots y_{i_{n-1}},\qquad
y_t y_{i_1}\cdots y_{i_{s_1}}
[y_{j_1},y_{j_2},\ldots,y_{j_{s_2}}].
\end{equation}
In the first family, $i_1<\cdots<i_{n-1}$.
In the second family, $s_1\geq0$, $ s_2\geq2$, $s_1+s_2=n-1$, $i_1<\cdots<i_{s_1}$
and  $j_1>j_2<\cdots<j_{s_2}.$ 

\begin{proposition}\label{sec6:prop-a2-even}
For every $n\geq2$, the polynomials in \eqref{sec6:a2-even}
are linearly independent modulo
$\textnormal{Id}^{\mathbb Z_2}(\mathcal A_2)$. Moreover, $c_{n,0}(\mathcal A_2)
=
n c_{n-1}(UT_2)
=
n\bigl(2^{n-2}(n-3)+2\bigr)$, for all $n\geq 2.$
\end{proposition}

\begin{proof}
Let $\mathcal H_1$ denote the set of admissible choices of indices
and block lengths in the second family of \eqref{sec6:a2-even}.
For $a\in\mathcal H_1$, set
\[
h_a=
y_t y_{i_1}\cdots y_{i_{s_1}}
[y_{j_1},y_{j_2},\ldots,y_{j_{s_2}}].
\]
For each $t$, let $i_1<\cdots<i_{n-1}$ be the remaining labels.
Suppose that
\[
f=
\sum_{t=1}^{n}
\lambda_t y_t y_{i_1}\cdots y_{i_{n-1}}
+
\sum_{a\in\mathcal H_1}\alpha_a h_a
\in\textnormal{Id}^{\mathbb Z_2}(\mathcal A_2).
\]

Fix $t$ and evaluate $y_t\mapsto e_{12}$ and
$y_r\mapsto e_{22}$ for $r\ne t$. A nonzero product must begin
with $y_t$ and every term containing a commutator vanishes.
Hence $f\mapsto\lambda_t e_{12}=0$, so all $\lambda_t$ vanish.

Now fix $a\in\mathcal H_1$ and consider the substitution $y_t\mapsto e_{12}$, $y_{j_1}\mapsto e_{23},$ $y_{i_r}\mapsto e_{22}+e_{33}$, for all $1\leq r\leq s_1$ and 
$y_{j_r}\mapsto e_{33}$ for all $2\leq r\leq s_2$. The selected polynomial $h_a$ evaluates to $e_{13}$. We show that it is the only polynomial with a nonzero value.

A nonzero term must begin with $y_t$ and $y_{j_1}$ must occur inside its commutator, since all other entries would be diagonal.
The variables evaluated at $e_{22}+e_{33}$ must lie outside the
commutator, since this element commutes with all remaining assigned
values. Conversely, those evaluated at $e_{33}$ must lie inside it,
since $e_{12}e_{33}=0$. Thus both blocks are determined. The
minimum-index convention in the commutator forces $y_{j_1}$ to be
its first entry and the increasing order of the other entries
determines $h_a$ uniquely. Therefore $f\mapsto\alpha_a e_{13}=0$,
proving independence.

For each of the $n$ choices of the first label $t$, the tail
ranges over the standard multilinear basis of $UT_2$ in
$n-1$ variables. Thus $c_{n,0}(\mathcal A_2)
=
n c_{n-1}(UT_2),$ for all $n\geq2.$
\end{proof}

\begin{itemize}
    \item {\bf Case $q=1$.} \end{itemize} 

Since $z_1y_2\in I$, every surviving monomial has its odd variable last. Applying the preceding reduction to its even factor, we obtain the following spanning family modulo $I$:
\begin{equation}\label{sec6:a2-oneodd}
y_t y_{i_1}\cdots y_{i_{n-2}}z_l,\qquad
y_t y_{i_1}\cdots y_{i_{s_1}}
[y_{j_1},y_{j_2},\ldots,y_{j_{s_2}}]z_l.
\end{equation}
In the first family, $i_1<\cdots<i_{n-2}$.
In the second family, $s_1\geq0$, $s_2\geq2$, $s_1+s_2=n-2$, $i_1<\cdots<i_{s_1}$ and $j_1>j_2<\cdots<j_{s_2}.$

\begin{proposition}\label{sec6:prop-a2-oneodd}
For every $n\geq3$, the polynomials in \eqref{sec6:a2-oneodd}
are linearly independent modulo
$\textnormal{Id}^{\mathbb Z_2}(\mathcal A_2)$. Moreover, $c_{n-1,1}(\mathcal A_2)=n(n-1)c_{n-2}(UT_2)$, for all $n\geq3.$
\end{proposition}

\begin{proof}
Fix the odd label $l$ and set $z_r=0$ for $r\ne l$.
A linear relation among the polynomials in \ref{sec6:a2-oneodd} then has the form
$hz_l$, where $h$ is a linear combination of the polynomials in
\ref{sec6:a2-even} on the remaining $n-1$ labels.

Apply the first evaluations of Proposition~\ref{sec6:prop-a2-even} to $h$ and set
$z_l\mapsto e_{24}$. They isolate the coefficients of the
commutator-free terms, since $e_{12}e_{24}=e_{14}$.
After these coefficients vanish, use the second evaluations of
that proposition and set $z_l\mapsto e_{34}$.
Since $e_{13}e_{34}=e_{14}$, each remaining coefficient is zero.
Repeating this for every $l$ proves independence.

There are $n(n-1)$ choices of the ordered pair $(t,l)$ and the
remaining factor ranges over the standard multilinear basis of
$UT_2$ in $n-2$ variables. This gives the stated formula.
\end{proof}

The preceding propositions together with equation (\ref{eqcodimensoes}) allow us to conclude the following.

\begin{theorem}\label{sec6:thm-a2}
For the superalgebra $\mathcal A_2$, the following statements hold.
\begin{enumerate}
    \item
    The polynomials in \eqref{sec6:a2-identities} generate
    $\textnormal{Id}^{\mathbb Z_2}(\mathcal A_2)$
    as a $T_{\mathbb Z_2}$-ideal.

    \item
    For every $n\geq3$, the polynomials
    in \eqref{sec6:a2-even} and \eqref{sec6:a2-oneodd} form
    a basis of $P_n^{\mathbb Z_2}$ modulo 
$P_n^{\mathbb Z_2}\cap
     \textnormal{Id}^{\mathbb Z_2}(\mathcal A_2)$.

    \item $c_n^{\mathbb Z_2}(\mathcal A_2)
    =
    2^{n-3}n(n^2-3n-2)+2n^2$, for all $n\geq 3$.
\end{enumerate}
\end{theorem}

Finally, by Remark \ref{reflectionidentities}, we obtain the following result for the superalgebra $\mathcal A_5$.

\begin{theorem}\label{sec6:thm-a5}
For the superalgebra $\mathcal A_5$, the following statements hold.
\begin{enumerate}
    \item $\textnormal{Id}^{\mathbb Z_2}(\mathcal A_5)
    =
    \langle
    [y_1,y_2][y_3,y_4]y_5, z_1z_2, y_1z_2
    \rangle_{T_{\mathbb Z_2}}.$

    \item $c_n^{\mathbb Z_2}(\mathcal A_5)
    =c_n^{\mathbb Z_2}(\mathcal A_2)=
    2^{n-3}n(n^2-3n-2)+2n^2$, for all $n\geq3$.
\end{enumerate}
\end{theorem}

\section{Graded identities of $\mathcal A_3$}

In this section, we determine generators for the graded identities of $\mathcal A_3$, describe multilinear bases for its relatively free graded algebra and compute its graded codimensions.

Recall that
\[
\begin{aligned}
(\mathcal A_3)_{0}
=F e_{22}+F e_{33}+F e_{12}+F e_{34}\quad \mbox{ and }\quad 
(\mathcal A_3)_{1}
=F e_{13}+F e_{14}+F e_{23}+F e_{24}.
\end{aligned}
\]
Thus, the following polynomials are graded identities of $\mathcal A_3$:
\begin{equation}\label{sec7:a3-identities}
\begin{gathered}
y_1[y_2,y_3]y_4,\quad
[y_1,y_2][y_3,y_4],\quad
\operatorname{St}_3(y_1,y_2,y_3),\quad 
z_1z_2,\quad
z_1[y_2,y_3]y_4,\quad
y_1[y_2,y_3]z_4.
\end{gathered}
\end{equation}

Let $I$ be the $T_{\mathbb Z_2}$-ideal generated by the polynomials
in \eqref{sec7:a3-identities}. Then
$I\subseteq\textnormal{Id}^{\mathbb Z_2}(\mathcal A_3).$ Every monomial with at least two odd variables contains a factor
$z_aUz_b$, with $U$ an even word and therefore belongs to $I$
since $z_1z_2\in I$. Consequently,
$P_{n-q,q}\subseteq I$ for all $2\leq q\leq n,$ and it suffices to consider the cases $q\in\{0,1\}$.

\begin{itemize}
    \item {\bf Case $q=0$.} \end{itemize} 

We first prove that, for $n\geq3$, any polynomial in $P_{n,0}$ can be written, modulo
$I$, as a linear combination of the monomials
\begin{equation}\label{sec7:a3-even}
y_t y_{i_1}\cdots y_{i_{n-1}},\qquad
y_{i_1}\cdots y_{i_{n-1}}y_t,
\end{equation}
where $i_1<\cdots<i_{n-1}$ and
$\{t,i_1,\ldots,i_{n-1}\}=\{1,\ldots,n\}$.
Above, the common monomial $y_1\cdots y_n$ is retained only once.

 For $n=3$, this follows from
$\operatorname{St}_3(y_1,y_2,y_3)\in I$.

Assume now that $n\geq4$. The identity $y_1[y_2,y_3]y_4$
allows us to order the internal variables of a monomial,
keeping its first and last variables fixed. Hence it suffices
to consider
\[
m=y_p y_{i_1}\cdots y_{i_{n-2}}y_q,
\qquad i_1<\cdots<i_{n-2}.
\]
For this fixed $n$, we argue by induction on $p$.
If $p<i_1$, the first $n-1$ variables are already increasingly
ordered, so $m$ belongs to \eqref{sec7:a3-even}. In particular,
this deals with $p=1$.

Suppose that $p>i_1$. Then
\[
m=y_{i_1}y_p y_{i_2}\cdots y_{i_{n-2}}y_q
  +[y_p,y_{i_1}]y_{i_2}\cdots y_{i_{n-2}}y_q.
\]
The variables following the commutator in the second term can be increasingly ordered modulo $I$ since $[y_1,y_2][y_3,y_4]\in I$.

Let $j_3<\cdots<j_n$ be the increasing arrangement of the labels
$i_2,\ldots,i_{n-2},q$. Expanding the commutator after this
reordering gives
\[
\begin{aligned}
m\equiv y_{i_1}y_p y_{i_2}\cdots y_{i_{n-2}}y_q
+y_p y_{i_1}y_{j_3}\cdots y_{j_n}
-y_{i_1}y_p y_{j_3}\cdots y_{j_n}
\pmod I.
\end{aligned}
\]

The first and third terms are covered by the induction hypothesis,
after ordering their internal variables. The second belongs to \eqref{sec7:a3-even} if $i_1<j_3$. Otherwise,
$j_3=q<i_1<j_4<\cdots<j_n$ and $y_1[y_2,y_3]y_4\in I$
allows us to interchange $y_{i_1}$ and $y_q$, since
$y_{j_4}\cdots y_{j_n}$ is nonempty. The resulting word also
belongs to \eqref{sec7:a3-even}, completing the induction.

\begin{proposition}\label{sec7:prop-a3-even}
For every $n\geq3$, the monomials in \eqref{sec7:a3-even} are linearly independent modulo
$\textnormal{Id}^{\mathbb Z_2}(\mathcal A_3)$. Moreover, $c_{n,0}(\mathcal A_3)=2n-1,$ for all $n\geq3.$
\end{proposition}

\begin{proof}
For each $t$, let $i_1<\cdots<i_{n-1}$ be the remaining labels, set
$L_t=y_ty_{i_1}\cdots y_{i_{n-1}}$ and $
R_t=y_{i_1}\cdots y_{i_{n-1}}y_t.$ Suppose that
\[
f=\sum_{t=1}^n\alpha_tR_t+\sum_{t=2}^n\beta_tL_t
\in\operatorname{Id}^{\mathbb Z_2}(\mathcal A_3).
\]
Evaluating $y_1\mapsto e_{34}$ and all other variables at $e_{33}$
leaves only $R_1$, so $\alpha_1=0$.

For $t\geq2$, evaluate $y_t\mapsto e_{12}$ and all other variables
at $e_{22}$. A nonzero word must begin with $y_t$. Among the $R_s$,
only $R_1$ does not begin with $y_1$ and its coefficient is already
zero. Hence $f\mapsto\beta_te_{12}=0$.

Finally, with all $\beta_t=0$, evaluating $y_s\mapsto e_{34}$ and
all other variables at $e_{33}$ gives
$f\mapsto\alpha_se_{34}=0$. Thus all coefficients vanish.
The two families, with their common word counted once, contain
$2n-1$ monomials.
\end{proof}

\begin{itemize}
    \item {\bf Case $q=1$.} \end{itemize} 

Write a monomial as $Uz_kV$, where $U$ and $V$ are possibly empty
even words. The identities $y_1[y_2,y_3]z_4$ and $z_1[y_2,y_3]y_4$
permit ordering the variables of $U$ except its first one and those
of $V$ except its last one, respectively. The surrounding factors
are absorbed into the graded substitutions or used as outside
multipliers. Therefore, for $n\geq3$, $P_{n-1,1}$ is spanned modulo $I$ by
\begin{equation}\label{eq:A3-one-odd-normal}
\begin{gathered}
y_ty_{i_1}\cdots y_{i_{n-2}}z_k,\qquad
z_ky_{i_1}\cdots y_{i_{n-2}}y_t,\qquad
y_ry_{i_1}\cdots y_{i_s}z_ky_{j_1}\cdots y_{j_{n-s-3}}y_t.
\end{gathered}
\end{equation}
The indices in each displayed internal even block are increasingly
ordered. In the third family, $0\leq s\leq n-3$.
These families are disjoint, according as the odd variable occurs
last, first, or internally.

\begin{proposition}\label{sec7:prop-a3-oneodd}
For every $n\geq3$, the monomials in \eqref{eq:A3-one-odd-normal}
are linearly independent modulo
$\operatorname{Id}^{\mathbb Z_2}(\mathcal A_3)$. Moreover,
$c_{n-1,1}(\mathcal A_3)
=2n(n-1)+n(n-1)(n-2)2^{n-3}$, for all $n\geq3.$
\end{proposition}

\begin{proof}
Suppose that a linear combination of the displayed monomials is a
graded identity. Fix the odd label $k$ and set every variable not
explicitly assigned a value below to zero.

For a monomial of the first family, evaluate
\[
y_t\mapsto e_{12},\qquad y_{i_a}\mapsto e_{22}\ (1\leq a\leq n-2),
\qquad z_k\mapsto e_{24}.
\]
A nonzero word must begin with $y_t$ and end with $z_k$.
Its ordered even block determines the selected monomial uniquely,
whose value is $e_{14}$. Thus its coefficient is zero.
For the second family, the analogous evaluation
\[
z_k\mapsto e_{13},\qquad y_{i_a}\mapsto e_{33}\ (1\leq a\leq n-2),
\qquad y_t\mapsto e_{34}
\]
forces the first and last variables to be $z_k$ and $y_t$, respectively,
and gives the same conclusion.

For a monomial of the third family, evaluate
\[
\begin{gathered}
y_r\mapsto e_{12},\qquad z_k\mapsto e_{23},\qquad y_t\mapsto e_{34},\\
y_{i_a}\mapsto e_{22}\ (1\leq a\leq s),\qquad
y_{j_b}\mapsto e_{33}\ (1\leq b\leq n-s-3).
\end{gathered}
\]
A nonzero value requires $y_r$ first and $y_t$ last, all variables
evaluated at $e_{22}$ before $z_k$ and all those evaluated at $e_{33}$
after it. These conditions and the ordering within each block select
precisely the chosen monomial, with value $e_{14}$. Its coefficient is
therefore zero. This also covers empty internal blocks and proves
independence.

Each of the first two families contains $n(n-1)$ monomials.
For the third, there are $n(n-1)(n-2)$ choices of $(r,k,t)$ and
each remaining label can be placed before or after $z_k$, with its
position within that block fixed by the increasing order. This gives
$n(n-1)(n-2)2^{n-3}$ monomials and proves the formula.

\end{proof}

The preceding results yield the following theorem.

\begin{theorem}\label{sec7:thm-a3}
For the superalgebra $\mathcal A_3$, the following statements hold.
\begin{enumerate}
\item
The polynomials in \eqref{sec7:a3-identities} generate
$\operatorname{Id}^{\mathbb Z_2}(\mathcal A_3)$ as a
$T_{\mathbb Z_2}$-ideal.
\item
For every $n\geq3$, the monomials in \eqref{sec7:a3-even} and
\eqref{eq:A3-one-odd-normal} form a basis of
$P_n^{\mathbb Z_2}$ modulo
$P_n^{\mathbb Z_2}\cap\operatorname{Id}^{\mathbb Z_2}(\mathcal A_3)$.
\item $c_n^{\mathbb Z_2}(\mathcal A_3)
=2n-1+n(n-1)\bigl(2+(n-2)2^{n-3}\bigr),$ for all $n\geq 3$.
\end{enumerate}
\end{theorem}

\section{Graded identities of $\mathcal A_4$ and $\mathcal A_7$}

In this section, we determine generators for the graded identities of
$\mathcal A_4$ and $\mathcal A_7$, describe multilinear bases for their
relatively free graded algebras and compute their graded codimensions.

We begin with $\mathcal A_4$. Recall that
\[
\begin{aligned}
(\mathcal A_4)_{0}
=F e_{22}+F e_{33}+F e_{12}+F e_{14}+F e_{24}\quad \mbox{ and }\quad 
(\mathcal A_4)_{1}
=F e_{13}+F e_{23}+F e_{34}.
\end{aligned}
\]
The polynomials
\begin{equation}\label{sec8:a4-identities}
y_1[y_2,y_3]y_4,\quad
z_1z_2z_3,\quad
z_1[y_2,y_3],\quad
z_1z_2y_3,\quad
y_1[y_2,y_3]z_4
\end{equation}
are graded identities of $\mathcal A_4$. 

Let $I$ be the $T_{\mathbb Z_2}$-ideal generated by the polynomials
in \eqref{sec8:a4-identities}. Then
$I\subseteq\textnormal{Id}^{\mathbb Z_2}(\mathcal A_4).$ Thus $P_{n-q,q}\subseteq I$ for all $3\leq q\leq n$ and it suffices to consider the cases $q\in\{0,1,2\}$.

\begin{itemize}
    \item \textbf{Case} $q=0$.
\end{itemize}

Since $y_1[y_2,y_3]y_4\in I$, every polynomial in $P_{n,0}$
is congruent modulo $I$ to a linear combination of the monomials
\begin{equation}\label{sec8:a4-even}
y_t y_{i_1}\cdots y_{i_{n-2}}y_p,
\end{equation}
where $i_1<\cdots<i_{n-2}$ and
$\{t,i_1,\ldots,i_{n-2},p\}=\{1,\ldots,n\}$.

\begin{proposition}\label{sec8:prop-a4-even}
For every $n\geq 4$, the monomials in \eqref{sec8:a4-even} are linearly
independent modulo $\textnormal{Id}^{\mathbb Z_2}(\mathcal A_4)$. Moreover, $c_{n,0}(\mathcal A_4)=n(n-1),$ for all $n\geq 4$.
\end{proposition}

\begin{proof}
For each ordered pair $(t,p)$ of distinct labels, let
$i_1<\cdots<i_{n-2}$ be the remaining labels. Suppose that
\[
f=\sum_{\substack{1\leq t,p\leq n\\t\ne p}}
\alpha_{t,p}y_ty_{i_1}\cdots y_{i_{n-2}}y_p
\in\textnormal{Id}^{\mathbb Z_2}(\mathcal A_4).
\]
Fix $(t,p)$ and evaluate $y_t\mapsto e_{12}$, $y_p\mapsto e_{24}$
and all remaining variables at $e_{22}$. A nonzero word must begin
with $y_t$ and end with $y_p$. Its ordered internal block therefore
selects the chosen monomial, whose value is $e_{14}$. Hence
$\alpha_{t,p}=0$ for every $(t,p)$, proving independence.
There are $n(n-1)$ choices of this ordered pair, which gives the formula.
\end{proof}

\begin{itemize}
    \item \textbf{Case} $q=1$.
\end{itemize}

Since $z_1[y_2,y_3]$ and $y_1[y_2,y_3]z_4$ belong to $I$,
every polynomial in $P_{n-1,1}$, with $n\geq2$, is congruent
modulo $I$ to a linear combination of the monomials
\begin{equation}\label{sec8:a4-oneodd}
z_k y_{i_1}\cdots y_{i_{n-1}},
\qquad
y_t y_{i_1}\cdots y_{i_s}z_k
y_{j_1}\cdots y_{j_{n-s-2}},
\qquad 0\leq s\leq n-2.
\end{equation}
The indices within each displayed even block are increasingly ordered. In both families, all indices are distinct and together form $\{1,\ldots,n\}$. The two internal even blocks in the second family may be empty. These families are disjoint, according as the first variable is odd or even.

\begin{proposition}\label{sec8:prop-a4-oneodd}
For every $n\geq3$, the monomials in \eqref{sec8:a4-oneodd}
are linearly independent modulo
$\textnormal{Id}^{\mathbb Z_2}(\mathcal A_4)$.
Moreover, $c_{n-1,1}(\mathcal A_4)=n+n(n-1)2^{n-2}$, for all $n\geq3$. 
\end{proposition}

\begin{proof}
Suppose that a linear combination of the displayed monomials
is a graded identity. Fix the odd label $k$. In each evaluation
below, all variables not explicitly assigned a value are sent
to zero.

For a monomial of the first family, evaluate $z_k\mapsto e_{13}$
and all its even variables at $e_{33}$. A nonzero word must
begin with $z_k$, since $e_{33}e_{13}=0$. The increasing order
of the remaining labels determines the chosen monomial
uniquely, with value $e_{13}$. Hence its coefficient is zero.

For a monomial of the second family, evaluate
\[
y_t\mapsto e_{12},\qquad
z_k\mapsto e_{23},\qquad
y_{i_a}\mapsto e_{22}\quad(1\leq a\leq s),\qquad
y_{j_b}\mapsto e_{33}\quad(1\leq b\leq n-s-2).
\]
A nonzero word must begin with $y_t$, with all variables
evaluated at $e_{22}$ before $z_k$ and all those evaluated
at $e_{33}$ after it. These conditions and the increasing
order within each block determine the chosen monomial
uniquely, with value $e_{13}$. Thus its coefficient is zero.
Varying the monomial and the odd label proves independence.

The first family contains $n$ monomials. For the second,
there are $n(n-1)$ choices of $(t,k)$, and each of the
remaining $n-2$ labels can occur before or after $z_k$,
with its position within that block fixed by the increasing
order. Hence the second family contains $n(n-1)2^{n-2}$
monomials, proving the formula.
\end{proof}

\begin{itemize}
    \item \textbf{Case} $q=2$.
\end{itemize}

Since $z_1[y_2,y_3]$, $z_1z_2y_3$, $z_1y_2z_3y_4$
and $y_1[y_2,y_3]z_4$ belong to $I$, every polynomial
in $P_{n-2,2}$, with $n\geq2$, is congruent modulo $I$
to a linear combination of the monomials
\begin{equation}\label{sec8:a4-twoodd}
z_j y_{i_1}\cdots y_{i_{n-2}}z_k,
\qquad
y_t y_{i_1}\cdots y_{i_s}z_j
y_{j_1}\cdots y_{j_{n-s-3}}z_k,
\qquad 0\leq s\leq n-3.
\end{equation}
The indices within each displayed even block are increasingly
ordered. In both families, all indices are distinct and together
form $\{1,\ldots,n\}$. The second family occurs only for $n\geq3$,
and its two internal even blocks may be empty. These families
are disjoint, according as the first variable is odd or even.
Both orders of the distinct odd labels $j,k$ are included.

\begin{proposition}\label{sec8:prop-a4-twoodd}
For every $n\geq3$, the monomials in \eqref{sec8:a4-twoodd}
are linearly independent modulo
$\textnormal{Id}^{\mathbb Z_2}(\mathcal A_4)$.
Moreover,
$c_{n-2,2}(\mathcal A_4)
=n(n-1)+n(n-1)(n-2)2^{n-3}$, for all $n\geq 3.$
\end{proposition}

\begin{proof}
Suppose that a linear combination of the displayed monomials
is a graded identity. Fix the ordered pair of odd labels $(j,k)$
and set $z_k\mapsto e_{34}$. In each evaluation below, all
variables not explicitly assigned a value are sent to zero.
Since $e_{34}\mathcal A_4=\{0\}$, every term in which $z_k$
is not last vanishes.

For a monomial of the first family, evaluate $z_j\mapsto e_{13}$
and all its even variables at $e_{33}$. Only the chosen
monomial survives, with value $e_{13}e_{34}=e_{14}$.
Hence its coefficient is zero.

For a monomial of the second family, apply the second
evaluation in the proof of
Proposition~\ref{sec8:prop-a4-oneodd} to its prefix preceding
$z_k$, with odd label $j$. That evaluation isolates the
chosen prefix, whose value is $e_{13}$. Thus only the
selected monomial survives after multiplication by
$z_k\mapsto e_{34}$, with value $e_{14}$.
Varying the monomial and the ordered pair of odd labels
proves independence.

The first family contains $n(n-1)$ monomials. For the second,
when $n\geq3$, there are $n(n-1)(n-2)$ choices of $(t,j,k)$,
and each of the remaining $n-3$ labels can occur before
or after $z_j$, with its position within that block fixed
by the increasing order. Hence the second family contains
$n(n-1)(n-2)2^{n-3}$ monomials, proving the formula.
\end{proof}

The preceding propositions allow us to conclude the following.

\begin{theorem}\label{sec8:thm-a4}
For the superalgebra $\mathcal A_4$, the following statements hold.
\begin{enumerate}
    \item
    The polynomials in \eqref{sec8:a4-identities} generate
    $\textnormal{Id}^{\mathbb Z_2}(\mathcal A_4)$ as a
    $T_{\mathbb Z_2}$-ideal.

    \item
    For every $n\geq3$, the monomials
    in \eqref{sec8:a4-even}, \eqref{sec8:a4-oneodd} and
    \eqref{sec8:a4-twoodd}
    form a basis of $P_n^{\mathbb Z_2}$ modulo $P_n^{\mathbb Z_2}\cap
     \textnormal{Id}^{\mathbb Z_2}(\mathcal A_4)$.

    \item $c_n^{\mathbb Z_2}(\mathcal A_4)
    =n\bigl(2n+n(n-1)2^{n-3}-1\bigr)$, for all $n\geq 3.$ 
\end{enumerate}
\end{theorem}

Finally, by Remark~\ref{reflectionidentities} and the preceding theorem, we obtain the following
result.

\begin{theorem}\label{sec8:thm-a7}
For the superalgebra $\mathcal A_7$, the following statements hold.
\begin{enumerate}
    \item
    The polynomials $y_1[y_2,y_3]y_4,$ $ z_1z_2z_3,$ $[y_1,y_2]z_3,$ $y_1z_2z_3$ and $ z_1[y_2,y_3]y_4$ generate $\textnormal{Id}^{\mathbb Z_2}(\mathcal A_7)$ as a $T_{\mathbb Z_2}$-ideal.

    \item $c_n^{\mathbb Z_2}(\mathcal A_7)=c_n^{\mathbb Z_2}(\mathcal A_4)
    =n\bigl(2n+n(n-1)2^{n-3}-1\bigr)$, for all  $n\geq3$.
\end{enumerate}
\end{theorem}

\section{Graded identities of $\mathcal A_6$}

In this section, we determine generators for the graded identities of $\mathcal A_6$, describe multilinear bases for its relatively free graded algebra and compute its graded codimensions.

Recall that
\[
\begin{aligned}
(\mathcal A_6)_{0}
=F e_{22}+F e_{33}+F e_{14}+F e_{23}\quad \mbox{ and }\quad 
(\mathcal A_6)_{1}
=F e_{12}+F e_{13}+F e_{24}+F e_{34}.
\end{aligned}
\]
The polynomials
\begin{equation}\label{sec9:a6-identities}
[y_1,y_2][y_3,y_4],\quad
y_1z_2y_3,\quad
y_1z_2z_3,\quad
z_1z_2y_3,\quad
z_1z_2z_3
\end{equation}
are graded identities of $\mathcal A_6$. 

Let $I$ be the $T_{\mathbb Z_2}$-ideal generated by the polynomials
in \eqref{sec9:a6-identities}. Thus,
$I\subseteq\textnormal{Id}^{\mathbb Z_2}(\mathcal A_6).$ Consequently,
$P_{n-q,q}\subseteq I$ for all $3\leq q\leq n,$ and it suffices to consider the cases $q\in\{0,1,2\}$.

\begin{itemize}
    \item \textbf{Case} $q=0$.
\end{itemize}

By Proposition~\ref{prop:UT2-central} and the standard multilinear reduction for
$UT_2$, every polynomial in $P_{n,0}$,
with $n\geq1$, is congruent modulo $I$ to a linear combination
of the polynomials
\begin{equation}\label{sec9:a6-even}
y_1\cdots y_n,\qquad
y_{i_1}\cdots y_{i_{s_1}}
[y_k,y_{j_1},\ldots,y_{j_{s_2}}].
\end{equation}
In the second family, $s_1\geq0$, $s_2\geq1$, $s_1+s_2=n-1,$ $i_1<\cdots<i_{s_1}$ and $
k>j_1<\cdots<j_{s_2}.$

\begin{proposition}\label{sec9:prop-a6-even}
For every $n\geq 3$, the polynomials in \eqref{sec9:a6-even}
are linearly independent modulo
$\textnormal{Id}^{\mathbb Z_2}(\mathcal A_6)$. Moreover, $c_{n,0}(\mathcal A_6)
=c_n(UT_2)=2^{n-1}(n-2)+2$ for all {$n\geq3.$}
\end{proposition}

\begin{proof}
Suppose that a linear combination of the displayed polynomials is a
graded identity. Evaluating all variables at $E=e_{22}+e_{33}$ annihilates every
commutator and gives $E$ on $y_1\cdots y_n$, so its coefficient is zero.

For a fixed polynomial of the second family, evaluate
\[
y_k\mapsto e_{23},\qquad
y_{i_a}\mapsto E\ (1\leq a\leq s_1),\qquad
y_{j_b}\mapsto e_{33}\ (1\leq b\leq s_2).
\]
A nonzero term must have $y_k$ inside its commutator.
The variables evaluated at $E$ must occur outside the commutator,
since $E$ commutes with all assigned elements. Conversely, those
evaluated at $e_{33}$ must occur inside it, since $e_{33}e_{23}=0$.
Thus both blocks are determined. The minimum-index convention forces
$y_k$ to be the first entry of the commutator and the increasing order
of the other entries determines the selected polynomial uniquely.
It evaluates to $e_{23}$, while every other commutator polynomial
vanishes. Its coefficient is therefore zero, proving independence.

For a commutator of length $r$, there are $\binom nr$ choices of its
variables and $r-1$ choices of its first entry. Hence
\[
c_{n,0}(\mathcal A_6)
=1+\sum_{r=2}^{n}\binom nr(r-1)
=2^{n-1}(n-2)+2=c_n(UT_2),\qquad n\geq 3.
\] \end{proof}

\begin{itemize}
    \item \textbf{Case} $q=1$.
\end{itemize}

Since $[y_2,y_3][y_4,y_5]$ and $y_1z_2y_3$ belong to $I$, any element of $P_{n-1,1}$ can be written, modulo $I$, as a linear combination of the following polynomials:
\begin{equation}\label{sec9:a6-oneodd}
\begin{gathered}
z_t y_{i_1}\cdots y_{i_{n-1}},\quad
y_{i_1}\cdots y_{i_{n-1}}z_t,\quad 
z_t y_{i_1}\cdots y_{i_{s_1}}
[y_k,y_{j_1},\ldots,y_{j_{s_2}}],\quad
y_{i_1}\cdots y_{i_{s_1}}
[y_k,y_{j_1},\ldots,y_{j_{s_2}}]z_t.
\end{gathered}
\end{equation}
In the first two families, $i_1<\cdots<i_{n-1}$.
In the last two families, $s_1\geq0$, $s_2\geq1$, $s_1+s_2=n-2$, $i_1<\cdots<i_{s_1}$ and $k>j_1<\cdots<j_{s_2}.$

\begin{proposition}\label{sec9:prop-a6-oneodd}
For every $n\geq 3$, the polynomials in \eqref{sec9:a6-oneodd}
are linearly independent modulo
$\textnormal{Id}^{\mathbb Z_2}(\mathcal A_6)$. Moreover, $c_{n-1,1}(\mathcal A_6)
=2n\,c_{n-1}(UT_2)
=2n\bigl(2^{n-2}(n-3)+2\bigr)$, for all $n\geq3.$
\end{proposition}

\begin{proof}
Suppose that a linear combination of these polynomials is a graded
identity. Fix the odd label $t$ and set $z_r=0$ for $r\ne t$.
The remaining relation has the form
$z_tf+gz_t$, where $f,g$ are linear combinations of the polynomials
in \eqref{sec9:a6-even} on the remaining $n-1$ labels.

Evaluating all even variables at $E$ and taking $z_t=e_{12}$ or
$z_t=e_{34}$ isolates, respectively, the commutator-free coefficients
of $f$ and $g$. Both therefore vanish.

For each remaining even polynomial, apply the separating evaluation
in the proof of Proposition~\ref{sec9:prop-a6-even}.
Taking $z_t=e_{12}$ isolates its coefficient in $f$, with value
$e_{12}e_{23}=e_{13}$, while the right-hand terms vanish because
$e_{23}e_{12}=0$. Taking $z_t=e_{34}$ similarly isolates its
coefficient in $g$, with value $e_{23}e_{34}=e_{24}$, since
$e_{34}e_{23}=0$. Thus all coefficients vanish.

There are $n$ choices of the odd label, two endpoint positions and
$c_{n-1}(UT_2)$ choices for the even polynomial. This gives the formula. \end{proof}

\begin{itemize}
    \item \textbf{Case} $q=2$.
\end{itemize}

Since $[y_1,y_2][y_3,y_4]$, $y_1z_2z_3$ and $z_1z_2y_3$ belong to $I$, any polynomial in $P_{n-2,2}$ can be written, modulo $I$, as a linear combination of the following polynomials:
\begin{equation}\label{sec9:a6-twoodd}
z_t y_{i_1}\cdots y_{i_{n-2}}z_l,\qquad
z_t y_{i_1}\cdots y_{i_{s_1}}
[y_k,y_{j_1},\ldots,y_{j_{s_2}}]z_l.
\end{equation}
In the first family, $i_1<\cdots<i_{n-2}$.
In the second family, $s_1\geq0$, $ s_2\geq1$, $s_1+s_2=n-3$, $i_1<\cdots<i_{s_1}$ and $
k>j_1<\cdots<j_{s_2}.$ 

\begin{proposition}\label{sec9:prop-a6-twoodd}
For every $n\geq3$, the polynomials in \eqref{sec9:a6-twoodd}
are linearly independent modulo
$\textnormal{Id}^{\mathbb Z_2}(\mathcal A_6)$. Moreover, $c_{n-2,2}(\mathcal A_6)
=n(n-1)c_{n-2}(UT_2)
=n(n-1)\bigl(2^{n-3}(n-4)+2\bigr),$ for all $n\geq3.$
\end{proposition}

\begin{proof}
Suppose that a linear combination of the displayed polynomials is a
graded identity. Fix an ordered pair $(t,l)$ of odd labels and set
$z_r=0$ for $r\notin\{t,l\}$.
Evaluate $z_t\mapsto e_{12}$, $z_l\mapsto e_{24}$ and all even
variables at $E$. Only the commutator-free term with odd order $(t,l)$
survives, with value $e_{14}$. Varying $(t,l)$ eliminates all such
coefficients.

For each remaining even polynomial, use the separating evaluation
from Proposition~\ref{sec9:prop-a6-even}, with
$z_t\mapsto e_{12}$ and $z_l\mapsto e_{34}$.
The selected term evaluates to $e_{12}e_{23}e_{34}=e_{14}$.
All other terms of the same odd order vanish by that proposition,
and terms of the opposite order vanish because $e_{34}e_{23}=0$.
Thus every coefficient is zero, proving independence.

For each of the $n(n-1)$ ordered pairs $(t,l)$, the even block ranges
over the standard multilinear basis of $UT_2$ in $n-2$ variables.
For $n\geq3$, this gives the stated formula. \end{proof}

The preceding propositions allow us to conclude the following.

\begin{theorem}\label{sec9:thm-a6}
For the superalgebra $\mathcal A_6$, the following statements hold.
\begin{enumerate}
    \item
    The polynomials in \eqref{sec9:a6-identities} generate
    $\textnormal{Id}^{\mathbb Z_2}(\mathcal A_6)$ as a
    $T_{\mathbb Z_2}$-ideal.

    \item
    For every $n\geq3$, the polynomials
    in \eqref{sec9:a6-even}, \eqref{sec9:a6-oneodd} and
    \eqref{sec9:a6-twoodd} form a basis of
    $P_n^{\mathbb Z_2}$ modulo $
    {P_n^{\mathbb Z_2}\cap
     \textnormal{Id}^{\mathbb Z_2}(\mathcal A_6)}$. 

    \item $c_n^{\mathbb Z_2}(\mathcal A_6)
    =2^{n-3}(n^3-n^2-4n-8)+2(n^2+n+1),$ , for all $n\geq 3.$  
\end{enumerate}
\end{theorem}

\section{Graded identities of $\mathcal A_8$}

In this section, we determine generators for the graded identities of $\mathcal A_8$, describe multilinear bases for its relatively free graded algebra and compute its graded codimensions.

Recall that
\[
\begin{aligned}
(\mathcal A_8)_{0}
=F e_{22}+F e_{33}+F e_{13}+F e_{24}\quad \mbox{ and }\quad 
(\mathcal A_8)_{1}
&=F e_{12}+F e_{14}+F e_{23}+F e_{34}.
\end{aligned}
\]
The polynomials
\begin{equation}\label{sec10:a8-identities}
\begin{gathered}
{}[y_1,y_2][y_3,y_4],\quad y_1[y_2,y_3]y_4,\quad
\operatorname{St}_3(y_1,y_2,y_3),\quad 
[y_1,y_2]z_3y_4,\quad z_1[y_2,y_3]y_4,\quad
y_1[y_2,y_3]z_4,\quad y_1z_2[y_3,y_4],\\
y_1z_2z_3y_4,\quad z_1[y_2,y_3]z_4,\quad
[y_1,y_2]z_3z_4,\quad z_1z_2[y_3,y_4],\quad 
z_1z_2z_3y_4,\quad y_1z_2z_3z_4,\quad z_1z_2z_3z_4
\end{gathered}
\end{equation}
are graded identities of $\mathcal A_8$.

Let $I$ be the $T_{\mathbb Z_2}$-ideal generated by the polynomials
in \eqref{sec10:a8-identities}. Thus,
$I\subseteq\textnormal{Id}^{\mathbb Z_2}(\mathcal A_8).$  Consequently,
$P_{n-q,q}\subseteq I$ for all $4\leq q\leq n,$
and it suffices to consider the cases $q\in\{0,1,2,3\}$.

\begin{itemize}
    \item \textbf{Case} $q=0$.
\end{itemize}

The first three polynomials in \eqref{sec10:a8-identities} are those
used in the reduction preceding Proposition~\ref{sec7:prop-a3-even}.
Hence, for $n\geq3$, $P_{n,0}$ is spanned modulo $I$ by
\begin{equation}\label{sec10:a8-even}
y_a y_{i_1}\cdots y_{i_{n-1}},\qquad
y_{i_1}\cdots y_{i_{n-1}}y_a,
\end{equation}
with the common word $y_1\cdots y_n$ retained only once.

\begin{proposition}\label{sec10:prop-a8-even}
For every $n\geq4$, the monomials in \eqref{sec10:a8-even}, with
$y_1\cdots y_n$ counted only once, are linearly independent modulo
$\textnormal{Id}^{\mathbb Z_2}(\mathcal A_8)$. Moreover, $c_{n,0}(\mathcal A_8)=2n-1,$ for all $ n\geq4.$
\end{proposition}

\begin{proof}
The linear map from $(\mathcal A_3)_0$ to $(\mathcal A_8)_0$ defined by
\[
e_{22}\mapsto e_{33},\qquad e_{12}\mapsto e_{13},\qquad
e_{33}\mapsto e_{22},\qquad e_{34}\mapsto e_{24}
\]
is an algebra isomorphism. Thus the two even components have the same
ordinary identities and the result follows from
Proposition~\ref{sec7:prop-a3-even}.
\end{proof}

\begin{itemize}
    \item \textbf{Case} $q=1$.
\end{itemize}

Write an arbitrary multilinear monomial with one odd variable
as $Uz_kV$, where $U,V$ are possibly empty even words.
If $V$ is empty, the identity $y_1[y_2,y_3]z_4$ allows us to
order all variables of $U$ except its first one. If $U$ is empty,
$z_1[y_2,y_3]y_4$ orders all variables of $V$ except its last one.
In each case, the adjacent transpositions are consequences of
the indicated identity, with the surrounding even words absorbed
into its substitutions or used as outside factors.

Using identities $y_1z_2[y_3,y_4]$ and $[y_1,y_2]z_3y_4$, we may assume that if both $U$ and $V$ are nonempty, all variables within each block
can be ordered increasingly. 
Thus, for $n\geq2$, $P_{n-1,1}$ is spanned modulo $I$ by
\begin{equation}\label{sec10:a8-oneodd}
\begin{gathered}
y_a y_{i_1}\cdots y_{i_{n-2}}z_k,\quad
 z_k y_{i_1}\cdots y_{i_{n-2}}y_a,\quad 
y_{i_1}\cdots y_{i_s}z_k y_{j_1}\cdots y_{j_t},
\quad s,t\geq1,\quad s+t=n-1.
\end{gathered}
\end{equation}
These families are disjoint, according as the odd variable occurs last, first, or internally.

\begin{proposition}\label{sec10:prop-a8-oneodd}
For every $n\geq4$, the monomials in \eqref{sec10:a8-oneodd} are linearly
independent modulo $\textnormal{Id}^{\mathbb Z_2}(\mathcal A_8)$.
Moreover, $c_{n-1,1}(\mathcal A_8)=n\bigl(2^{n-1}+2n-4\bigr),$ , for all $n\geq 4.$
\end{proposition}

\begin{proof}
Suppose that a linear combination of these monomials is a graded
identity and fix a monomial with odd label $k$.
For the first family, evaluate $y_a\mapsto e_{13}$,
$z_k\mapsto e_{34}$ and all remaining even variables at $e_{33}$.
A nonzero word must begin with $y_a$ and end with $z_k$.
Its ordered middle block selects the chosen monomial, with value
$e_{14}$, so its coefficient is zero.
For the second family, the evaluation $z_k\mapsto e_{12}$,
$y_a\mapsto e_{24}$ and all remaining even variables at $e_{22}$
gives the same conclusion.

For the third family, evaluate
\[
y_{i_r}\mapsto e_{22}\ (1\leq r\leq s),\qquad
z_k\mapsto e_{23},\qquad
y_{j_r}\mapsto e_{33}\ (1\leq r\leq t).
\]
A nonzero word must place every $e_{22}$ before $e_{23}$ and every
$e_{33}$ after it. Since both blocks are nonempty and ordered, only
the chosen monomial survives, with value $e_{23}$. Thus all
coefficients vanish.

Each endpoint family contains $n(n-1)$ monomials. In the third,
choose the odd label and the nonempty proper subset of the remaining
labels occurring before it. Hence, for $n\geq3$,
\[
c_{n-1,1}(\mathcal A_8)
=2n(n-1)+n(2^{n-1}-2)
=n\bigl(2^{n-1}+2n-4\bigr).
\]
\end{proof}

\begin{itemize}
    \item \textbf{Case} $q=2$.
\end{itemize}

Write an arbitrary multilinear monomial with two odd variables
as $Uz_iVz_kW$, where $U,V,W$ are possibly empty even words.
If both $U$ and $W$ are nonempty, this monomial belongs to $I$, since $y_1z_2z_3y_4\in I$.

The variables within each of the three even blocks can be
increasingly ordered modulo $I$, using the identities $[y_1,y_2]z_3z_4$, $z_1[y_2,y_3]z_4$ and $z_1z_2[y_3,y_4]$. 
Thus, for $n\geq3$, $P_{n-2,2}$ is spanned modulo $I$ by
\begin{equation}\label{sec10:a8-twoodd}
\begin{aligned}
&z_i y_{i_1}\cdots y_{i_s}z_k y_{j_1}\cdots y_{j_t},
&&s,t\geq0,\quad s+t=n-2,\\
&y_{i_1}\cdots y_{i_s}z_i y_{j_1}\cdots y_{j_t}z_k,
&&s\geq1,\quad t\geq0,\quad s+t=n-2.
\end{aligned}
\end{equation}
The families are disjoint, since their first variables have different
parities. Both orders of the distinct odd labels $i,k$ are included.

\begin{proposition}\label{sec10:prop-a8-twoodd}
For every $n\geq4$, the monomials in \eqref{sec10:a8-twoodd}
are linearly independent modulo
$\textnormal{Id}^{\mathbb Z_2}(\mathcal A_8)$. Moreover, $c_{n-2,2}(\mathcal A_8)
=n(n-1)(2^{n-1}-1),$ for all $n\geq4.$
\end{proposition}

\begin{proof}
Suppose that a linear combination of the displayed monomials is a
graded identity and fix one of them.
For a monomial of the first family, evaluate
\[
z_i\mapsto e_{12},\quad z_k\mapsto e_{23},\quad
y_{i_r}\mapsto e_{22}\ (1\leq r\leq s),\quad
y_{j_r}\mapsto e_{33}\ (1\leq r\leq t).
\]
A nonzero word must begin with $z_i$, place every $e_{22}$ between
$z_i$ and $z_k$ and place every $e_{33}$ after $z_k$.
The ordering of the blocks therefore selects the chosen monomial,
with value $e_{13}$.

For a monomial of the second family, use the same even assignments
and take $z_i\mapsto e_{23}$, $z_k\mapsto e_{34}$.
A nonzero word must place every $e_{22}$ before $z_i$, every $e_{33}$
between $z_i$ and $z_k$ and $z_k$ last. Since $s\geq1$, such a word
belongs to the second family and is uniquely determined.
Its value is $e_{24}$. Both evaluations also eliminate the opposite
odd order, so every coefficient vanishes.

For each of the $n(n-1)$ ordered pairs $(i,k)$, the first family
contains $2^{n-2}$ monomials, while the second contains $2^{n-2}-1$,
since its initial block is nonempty. For $n\geq3$, their sum gives
\[
c_{n-2,2}(\mathcal A_8)
=n(n-1)\bigl(2^{n-2}+2^{n-2}-1\bigr)
=n(n-1)(2^{n-1}-1).
\]
\end{proof}

\begin{samepage}
\begin{itemize}
    \item \textbf{Case} $q=3$.
\end{itemize}

Write a multilinear monomial with three odd variables as
$Uz_iVz_jWz_kT$, where $U,V,W,T$ are possibly empty even words.
If $U$ or $T$ is nonempty, the monomial belongs to $I$ by
$y_1z_2z_3z_4$ or $z_1z_2z_3y_4$, respectively. 
Therefore it suffices to consider monomials of type $z_iVz_jWz_k$.
\end{samepage}

The identity $z_1[y_2,y_3]z_4\in I$ allows us to order the even variables increasingly within each of the two blocks. Hence, for $n\geq 3$, $P_{n-3,3}$ is spanned modulo $I$
by the monomials
\begin{equation}\label{sec10:a8-threeodd}
z_i y_{i_1}\cdots y_{i_s}z_j y_{j_1}\cdots y_{j_t}z_k,
\qquad s,t\geq0,\quad s+t=n-3.
\end{equation}
All orders of the three distinct odd labels $i,j,k$ are included.

\begin{proposition}\label{sec10:prop-a8-threeodd}
For every $n\geq4$, the monomials in \eqref{sec10:a8-threeodd} are
linearly independent modulo
$\textnormal{Id}^{\mathbb Z_2}(\mathcal A_8)$. Moreover, $c_{n-3,3}(\mathcal A_8)=n(n-1)(n-2)2^{n-3},$ , for all $n\geq 4.$
\end{proposition}

\begin{proof}
Suppose that a linear combination of these monomials is a graded
identity. For a fixed monomial, evaluate
\[
\begin{gathered}
z_i\mapsto e_{12},\qquad z_j\mapsto e_{23},\qquad z_k\mapsto e_{34},\\
y_{i_r}\mapsto e_{22}\ (1\leq r\leq s),\qquad
y_{j_r}\mapsto e_{33}\ (1\leq r\leq t).
\end{gathered}
\]
A nonzero word must have odd order $(i,j,k)$, with all variables
assigned $e_{22}$ between $z_i,z_j$ and those assigned $e_{33}$
between $z_j,z_k$. The ordering of the two blocks therefore selects
the chosen monomial uniquely, with value $e_{14}$.
Its coefficient is zero, proving independence.

There are $n(n-1)(n-2)$ choices of the ordered odd labels and each
remaining label can be placed in either even block. This gives
$n(n-1)(n-2)2^{n-3}$ monomials and proves the formula.
\end{proof}

The preceding propositions allow us to conclude the following.

\begin{theorem}\label{sec10:thm-a8}
For the superalgebra $\mathcal A_8$, the following statements hold.
\begin{enumerate}
    \item The polynomials in \eqref{sec10:a8-identities} generate
    $\textnormal{Id}^{\mathbb Z_2}(\mathcal A_8)$ as a
    $T_{\mathbb Z_2}$-ideal.

    \item For every $n\geq4$, the monomials in
    \eqref{sec10:a8-even}, \eqref{sec10:a8-oneodd},
    \eqref{sec10:a8-twoodd} and \eqref{sec10:a8-threeodd} form a
    basis of $P_n^{\mathbb Z_2}$ modulo
    $P_n^{\mathbb Z_2}\cap\textnormal{Id}^{\mathbb Z_2}(\mathcal A_8)$.

    \item  $c_n^{\mathbb Z_2}(\mathcal A_8)
    =2^{n-3}n(n^2+n+2)+n(n-1)-1,$ for all $n\geq 4.$
\end{enumerate}
\end{theorem}

We finish by distinguishing the eight superalgebras considered above.

\begin{remark}\label{sec10:remark-distinct-gradings}
The superalgebras $\mathcal A_1,\ldots,\mathcal A_8$ are pairwise
non-PI-equivalent as $\mathbb Z_2$-graded algebras. Consequently,
they are pairwise nonisomorphic as graded algebras
\end{remark}

\vspace{0.3cm}
 \textbf{Acknowledgments}
\vspace{0.3cm}

The authors would like to express their gratitude to Professor F. Yasumura for suggesting the approach used in the proof of Theorem~\ref{gradingond0}.

\end{document}